\documentclass{amsart}

\usepackage{amsfonts}
\usepackage{amsthm}
\usepackage{amsmath}
\usepackage{amsfonts}
\usepackage{latexsym}
\usepackage{amssymb}
\usepackage{amscd}
\usepackage[latin1]{inputenc}
\usepackage{verbatim}
\usepackage[hypertex]{hyperref}

\newtheorem{theorem}{Theorem}[section]
\newtheorem{lemma}[theorem]{Lemma}
\newtheorem{proposition}[theorem]{Proposition}
\newtheorem{corollary}[theorem]{Corollary}

\newtheorem{definition}[theorem]{Definition}

\theoremstyle{definition}

\theoremstyle{remark}

\def\cF{\mathcal{F}}

\def\mloc{M_{\rm loc}\,}
\def\mlock{M_{\rm loc}^{[k]}\,}
\def\mlockplusone{M_{\rm loc}^{[k+1]}\,}
\def\mlockminusone{M_{\rm loc}^{[k-1]}\,}
\def\mloczero{M_{\rm loc}^{[0]}\,}
\def\mlockzero{M_{\rm loc}^{[k_0]}\,}
\def\mloctwo{M_{\rm loc}^{[2]}\,}
\def\mloctwoplusk{M_{\rm loc}^{[2+k]}\,}

\title[Injective Envelopes and Local Multiplier Algebras, II]{Injective Envelopes and Local Multiplier Algebras of Some Spatial Continuous Trace C$^*$-algebras, II}

\author{Mart\'in Argerami}
\address{Department of Mathematics and Statistics, University of Regina,
Regina, Saskatchewan S4S 0A2, Canada}

\author{Douglas Farenick}
\address{Department of Mathematics and Statistics, University of Regina,
Regina, Saskatchewan S4S 0A2, Canada}

\author{Pedro Massey}
\address{Departamento de Matem\'atica, Facultad de Ciencias Exactas,  Universidad Nacional de La Plata
and Instituto Argentino de Matem\'atica--CONICET,
Buenos Aires, Argentina}


\begin{document}
\begin{abstract}
We determine the injective envelope and local multiplier algebra of a continuous trace
C$^*$-algebra $A$ that arises from a continuous Hilbert bundle over an arbitrary locally
compact Hausdorff space. In addition, we show that the second-order local multiplier
algebra $\mloctwo(A)$ of any such algebra $A$ is injective.
\end{abstract}
\maketitle
%%%%%%%%%%%%%%%%%%%%%%%%%%%%
\section*{Introduction}
The injective envelope $I(A)$ of a C$^*$-algebra $A$ \cite{hamana1979a} provides a useful
ambient C$^*$-algebra in which one can analyse the multipliers of essential
ideals of $A$. In fact the local multiplier algebra $\mloc(A)$ of $A$
\cite{Ara--Mathieu-book} can be obtained from the injective envelope of $A$ by
considering the C$^*$-subalgebra of all $x\in I(A)$ for which $x$ is a norm-limit of a
sequence $x_n\in M(I_n)$ for various essential ideals $I_n$ of $A$
\cite{frank--paulsen2003}. However, $I(A)$ and $\mloc(A)$ are difficult to determine
precisely, even if one has extensive knowledge about $A$ itself. Indeed, on page 55 of \cite{blecher2007},
D.~Blecher writes, ``Thus the injective envelope is
mostly useful as an abstract tool because of the properties it possesses; one cannot hope
to concretely be able to say what it is.'' Anyone who has worked with injective envelopes
will find this comment completely understandable. Nevertheless, in this paper we
determine explicitly (Theorem \ref{inj env}) the injective envelope of a continuous trace C$^*$-algebra $A$ 
of the spatial type considered by Fell \cite{fell1961}.
We then use the embedding of the local multiplier algebra of
$A$ into its injective envelope to prove that the second-order local multiplier algebras
of such $A$ are injective (Theorem \ref{mloc-2}). An immediate consequence of this last result is
that the second-order local multiplier algebra of
$C_0(T)\otimes\mathbb K$ is injective for every locally compact Hausdorff space $T$, a fact which was known previously
to hold only under certain assumptions about the topology of $T$ \cite{ara-mathieu2010,somerset2000}.
The results of this paper complete 
a line of investigation
that started with \cite{argerami--farenick--massey2009} and was continued in \cite{argerami--farenick--massey2010}. 

In the case of an abelian C$^*$-algebra $A=C_0(T)$, to determine the local multiplier algebra and the injective envelope of $A$
one must pass from $T$ to a Stonean space
obtained from $T$ by performing an inverse limit  $\Delta=\displaystyle\lim_{\leftarrow}\beta U$
of Stone--\v Cech compactifications $\beta U$ of dense open subsets $U$ of $T$.
This passage from $T$ to a Stonean space $\Delta$ cannot be avoided if one aims to compute explicitly the enveloping C$^*$-algebras
$\mloc(A)$ and $I(A)$
in the case of arbitrary continuous trace C$^*$-algebras $A$ of the type studied by Fell.
In \cite{argerami--farenick--massey2010} we determined $\mloc(A)$ and $I(A)$ for those $A$
in which the spectrum $T=\hat A$ was assumed to be a Stonean space. In this paper we make the passage from an
arbitrary locally compact Hausdorff space $T$ to its projective hull
$\Delta=\displaystyle\lim_{\leftarrow}\beta U$. Doing so entails the determination of a new continuous Hilbert bundle over $\Delta$ which is obtained as
a direct limit of continuous Hilbert bundles over the compact spaces $\beta U$ for all dense open subsets $U$ of $T$.
Because the essential ideals of $A$ are parameterised by dense open subsets of $T$, this direct limit of
Hilbert bundles also induces a direct limit $A^\Delta$ of continuous trace C$^*$-algebras. 
It is through these limiting bundles and algebras that we obtain our main results.

The paper is organised as follows. In Section \S\ref{S:prelim} the algebras and
structures under study are introduced. The three subsequent sections treat the limiting
processes of Hilbert bundles and C$^*$-algebras that accompany the passage from
$T$ to $\Delta=\displaystyle\lim_{\leftarrow}\beta U$. In particular, in Section \S\ref{S:ideal}
we note the bijective correspondence between essential ideals $I$ of $A$ and dense
open subsets $X^I\subset T$ and represent each
essential ideal $I$ of $A$ as an essential ideal of a spatial continuous trace C$^*$-algebra
$A^I$ with spectrum $\beta X^I$. Section \S\ref{S:hb} constructs a continuous Hilbert bundle $\Omega^\Delta$
over the Stonean space $\Delta$ by way of a direct limit of Hilbert bundles $\Omega^I$ over
$\beta X^I$ affiliated with each essential ideal $I$ of $A$. In Section \S\ref{S:limit algs} we construct a direct
limit C$^*$-algebra $\displaystyle\lim_\rightarrow A^I$ and in Section \S\ref{sect:chain of inclusions} we show that
$A\subset\displaystyle\lim_\rightarrow A^I\subset \mloc(A)$. The main results concerning the
determination of the injective envelope of $A$ and the injectivity of the second-order local multiplier
algebra $\mloctwo(A)$ are given in Section \S\ref{S:main results}.

%%%%%%%%%%%%%%%%%%%%%%%%%%%%
\section{Preliminaries}\label{S:prelim}

When referring to ideals of a C$^*$-algebra, we shall always mean ideals which are closed in the norm topology.
The term homomorphism is understood to be with respect to the category of C$^*$-algebras, meaning that
homomorphisms of C$^*$-algebras are $*$-homomorphisms, and are unital homomorphisms if the algebras involved are unital.

%%%%
\subsubsection*{Essential ideals and local multiplier algebras}

Recall that an ideal $K$ of a C$^*$-algebra $A$ is an \emph{essential ideal} if
$K\cap J\neq\{0\}$ for every nonzero ideal $J$ of $A$.

Let  $\mathcal I_{\rm ess}(A)$ be the set of all essential ideals of $A$, which we consider
as a directed set under the partial order $ \preccurlyeq$ defined by
$ J \preccurlyeq I$ {if and only if} $I\subset J$.

For each $I\in \mathcal I_{\rm ess}(A)$, let $M(I)$ denote its {multiplier algebra}. If
$I,J\in \mathcal I_{\rm ess}(A)$ are such that
$I\subset J$, then there is a unique monomorphism
\begin{equation}\label{ma-incl}
\varrho_{JI}:M(J)\rightarrow M(I)\quad\mbox{such that}\quad
\iota_I=\varrho_{JI}\circ\iota_{J}{}_{\vert I}\,,
\end{equation}
where $\iota_K:K\rightarrow M(K)$
denotes the canonical embedding of $K$ into $M(K)$. Hence,
$(\mathcal I_{\rm ess}(A), \{M(I)\}_I, \{\varrho_{JI}\}_{J \preccurlyeq I})$ is a direct
system of C$^*$-algebras and monomorphisms, and the direct limit C$^*$-algebra
of this system is denoted by
\[
\mloc(A)\;=\;\lim_{\rightarrow}\,M(I)\,.
\]
The C$^*$-algebra $\mloc(A)$ is called the \emph{local multiplier algebra} of $A$.

One can consider the local multiplier algebra of $\mloc(A)$, and so forth, thereby yielding
higher order local multiplier algebras. So we write
\[
\mlock(A)\;=\;\mloc\left(\mlockminusone(A)\right) \,,\;\forall\,k\in\mathbb N\,,
\]
where $\mloczero(A)$ is taken to be $A$. Although
very little is known about the sequence $\{\mlock(A)\}_{k\in\mathbb N}$, it is known that the sequence becomes constant if for some $k_0$
the C$^*$-algebra $\mlockzero(A)$ is an AW$^*$-algebra---for in this case, $\mlock(A)=\mlockzero(A)$ for every $k\geq k_0$
\cite[Theorem 2.3.8]{Ara--Mathieu-book}. Only relatively recently has it been discovered
\cite{ara-mathieu2006,ara-mathieu2008,argerami--farenick--massey2009} that $\mloctwo(A)$ need not coincide with $\mloc(A)$,
and the reasons for this gap are just now starting to be understood \cite{ara-mathieu2010}.

%%%%%%%%%%%%
\subsubsection*{Injective envelopes}

An \emph{injective C$^*$-algebra} is a unital C$^*$-algebra $C$
with the property that, for any triple $(B,D, \kappa)$ of unital C$^*$-algebras $B$, $D$
and unital completely isometric linear map $\kappa:B\rightarrow D$, every unital completely positive (ucp) linear
map $\phi:B\rightarrow C$ extends to a ucp
$\Phi:D\rightarrow C$ such that $\phi=\Phi\circ\kappa$ \cite[\S IV.2]{Blackadar-book}.
If $A$ is an arbitrary C$^*$-algebra, not necessarily unital, then an \emph{injective envelope} of $A$ is a pair $(C,\alpha)$
such that $C$ is an injective C$^*$-algebra, $\alpha:A\rightarrow C$ is a monomorphism which is assumed to be unital if $A$ is unital,
with the property that if $\tilde C$ is an injective C$^*$-algebra with $\alpha(A)\subset\tilde C\subset C$,
then $\tilde C=C$.
Every C$^*$-algebra has an injective envelope, and any two injective envelopes $(C,\alpha)$ and $(C_1,\alpha_1)$
of $A$ are
isomorphic by an isomorphism $\varphi:C\rightarrow C_1$ for which $\varphi\circ\alpha=\alpha_1$ \cite{hamana1979a}.

Thus, we may refer generically to ``the'' injective envelope of $A$, which we denote by $I(A)$.
The injective envelope of $A$ and the local multiplier algebras of $A$ are related by way of the
C$^*$-algebra inclusions
\begin{equation}\label{fp inclusion}
A\;\subset \;\mlock(A)\;\subset\; \mlockplusone(A)\;\subset\; I(A)\,,\;\forall\,k\in\mathbb N\,,
\end{equation}
where the inclusions are as unital C$^*$-subalgebras, except for the first inclusion if $A$ is nonunital.
These inclusions are uniquely determined by the inclusion (embedding) $\alpha:A\rightarrow I(A)$ of $A$ in $I(A)$.
More explicitly, $\mloc(A)$ is the closure in $I(A)$ of the union of all the idealizers in $I(A)$ of all essential
ideals of $A$
\cite{frank--paulsen2003}.

%%%%%%%%%%%%
\subsubsection*{C$^*$-modules}

The Hilbert C$^*$-modules \cite[\S II.7]{Blackadar-book}
that we use are left modules $E$ over an abelian C$^*$-algebra $Z$.
Recall that $B(E)$ denotes the C$^*$-algebra of bounded, adjointable endomorphisms of $E$ and
$K(E)$ denotes the set of compact elements of $B(E)$---namely, the norm closure of the linear
space $\mathcal F(E)$ of all elements (called \emph{finite-rank} endomorphisms) obtained through finite
sums of endomorphisms of the form $\Theta_{\omega,\nu}$, where $\omega,\nu\in E$ and
$\Theta_{\omega,\nu}\xi\,=\,\langle\xi,\nu\rangle\cdot\omega$, for all $\xi\in E$. The pertinent facts
we require are: $K(E)$ is an essential ideal of $B(E)$ and $B(E)$ is the multiplier algebra of $K(E)$.
We will also use the fact that $\mathcal F(E)$ is a left $Z$-module via $f\cdot \Theta_{\omega,\nu}=\Theta_{f\cdot\omega,\nu}$,
for $f\in Z$.

 %%%
\subsubsection*{Topology}

Throughout we shall assume that $T$ denotes a locally compact Hausdorff space.
As usual,
$C(T)$, $C_b(T)$, and $C_0(T)$ denote, respectively, the involutive algebras
of all continuous complex-valued functions on $T$, all bounded $f\in C(T)$, and all
$f\in C(T)$ that vanish at infinity respectively.

\subsubsection*{Vector and operator fields}

Assume that 
$(T, \{H_t\}_{t\in T})$ and $(T, \{B(H_t)\}_{t\in T})$ are fibred spaces 
where each $H_t$ is a Hilbert space.
A cross section of $(T, \{H_t\}_{t\in T})$ is a vector field
$\nu:T\rightarrow \bigsqcup_t\,H_t$ in which $\nu(t)\in H_t$,
for every $t\in T$.  Likewise, a cross section of 
$(T, \{B(H_t)\}_{t\in T})$ is an operator field
$x:T\rightarrow \bigsqcup_t\,B(H_t)$ such that $x(t)\in B(H_t)$,
for every $t\in T$.

For such cross sections $\nu$, $x$, we define functions $\check{\nu},\check{x}:T\rightarrow\mathbb R$ by
\[
\check{\nu}\,(t)\;=\;\|\nu(t)\|\,,\quad \check{x}\,(t)\;=\;\|x(t)\|\,.
\]
We say that $\nu$ is bounded if $\sup_{t\in T}\check{\nu}(t)\,<\,\infty$.
The boundedness of $x$ is defined analogously.

A {\em continuous Hilbert bundle} \cite{dixmier--douady1963} is a triple $(T,\{H_t\}_{t\in T},\Omega)$,
where
$\Omega$ is a set of vector fields on $T$ with fibres $H_t$
such that:
\begin{enumerate}
\item[{(I)}] $\Omega$ is a $C(T)$-module with the action $(f\cdot\omega)(t)=f(t)\omega(t)$;
\item[{(II)}] for each $t\in T$, $\{\omega(t):\ \omega\in \Omega\}=H_{t}$;
\item[{(III)}] $\check{\omega}\in C(T)$, for all $\omega\in\Omega$;
\item[{(IV)}] $\Omega$ is closed under local uniform approximation---that is, if $\xi:T\rightarrow
                \bigsqcup_t\,H_t$
                is any vector field such that for every $t_0\in T$ and $\varepsilon>0$ there is an
                open set $U\subset T$
                containing $t_0$ and a $\omega\in \Omega$ with $\|\omega(t)-\xi(t)\|<\varepsilon$
                for all $t\in U$, then necessarily $\xi\in\Omega$.
\end{enumerate}

Given a continuous Hilbert bundle $(T,\{H_t\}_{t\in T},\Omega)$, let
\begin{equation}\label{def Omega_0_b }
\Omega_b\;=\;\{\omega\in\Omega\,:\,\check{\omega}\in C_b(T)\} \quad\mbox{and}\quad
\Omega_0\;=\;\{\omega\in\Omega\,:\,\check{\omega}\in C_0(T)\} \,.
\end{equation}

It is easy to see that $\Omega_b$ and $\Omega_0$ are Hilbert C$^*$-modules over
$C_b(T)$ and $C_0(T)$ respectively, where the inner product $\langle\omega_1,\omega_2\rangle$
of $\omega_1,\omega_2\in\Omega$ is the continuous function
\[
\langle\omega_1,\omega_2\rangle\,(t)\;=\;\langle\omega_1(t),\omega_2(t)\rangle\,,\;t\in T\,.
\]

%%%%%%%
\subsubsection*{Spatial continuous trace C$^*$-algebras}

We now describe the class of C$^*$-algebras of interest in this paper.

Assume that $(T,\{H_t\}_{t\in T},\Omega)$
is a continuous Hilbert bundle.
An operator field $a$ is
\emph{almost finite-dimensional} with respect to this bundle
 if for each $t_0\in T$ and $\varepsilon>0$
       there exist an open set $U\subset T$ containing $t_0$
       and $\omega_1,\dots, \omega_n\in \Omega$ such that
          \begin{enumerate}
            \item[{(a)}] $\omega_1(t),\dots,\omega_n(t)$ are linearly independent for every $t\in U$, and
             \item[{(b)}] $\|p_ta(t)p_t-a(t)\|<\varepsilon$ for all $t\in U$, where $p_t\in B(H_t)$ is the
                         projection with range $\mbox{\rm Span}\,\{\omega_j(t)\,:\,1\le j\le n\}$.
          \end{enumerate}
 Moreover, $a$ is  \emph{weakly continuous} if the complex-valued function
\[
t\;\mapsto\;\langle a(t)\omega_1(t), \omega_2(t)\rangle
\]
is continuous for every $\omega_1,\omega_2\in\Omega$.

We denote by
$A=A(T,\{H_t\}_{t\in T},\Omega)$
the C$^*$-algebra,
with respect to pointwise operations and norm $\|a\|=\max\{\|a(t)\|\,:\,t\in T\}$, of
all weakly continuous almost finite-dimensional operator fields $a$ for which $\check{a}\in C_0(T)$.
Such C$^*$-algebras $A$
were studied by Fell \cite{fell1961}, and he proved that each such $A$
is a continuous trace C$^*$-algebra with spectrum $\hat A\simeq T$ \cite[Theorems 4.4, 4.5]{fell1961}.
We call the algebra $A$ the \emph{Fell}, or the \emph{spatial},
continuous trace C$^*$-algebra associated with the Hilbert bundle
$(T,\{H_t\}_{t\in T},\Omega)$.

%%%%%%%%%%%%%%%%%%%%%%%%%%%%%%%%%
\section{Extended Representations of Essential Ideals }\label{S:ideal}

 Let  $(T,\{H_t\}_{t\in T},\Omega)$ be a continuous Hilbert bundle over a locally compact Hausdorff space $T$.
Suppose that $I$ is an arbitrary ideal of $A\;=\;A(T,\{H_t\}_{t\in T},\Omega)$.
In this section we shall construct a continuous Hilbert bundle $(\beta X^I, \{H_t^I\}_{t\in \beta X^I},\Omega^{I})$
over the Stone--\v Cech compactification $\beta X^I$ of $X^I$.
Moreover, if we let $A^I$ be the Fell continuous C$^*$-algebra associated with this bundle we shall show that $I$ embeds into $A^I$ as an essential ideal.

Let $Z^I\subset T$ denote the closed set
\[
Z^I\;=\;\{t\in T\,:\,b(t)=0\,,\;\forall\,b\in I\}\,,
\]
and let $X^I$ be the open set $X^I=T\setminus X^I$. The open set $X^I$
is homeomorphic to both the primitive ideal space $\mbox{Prim}\, I $ and to the spectrum  $\hat I$
of $I$ \cite[Proposition A.27]{Raeburn--Williams-book}. Moreover,
$I$ is an essential ideal of $A$ if and only if $X^I$ is dense in $T$.

Recall that $\Omega_b\;=\;\{\omega\in\Omega\,: \check{\omega}\in C_b(T)\}$ is a $C_b(T)$-module, and define a normed vector space
$\Omega_b|_{X^I}$ of bounded restricted
vector fields by
 \begin{equation}\label{bounded restriction}
\Omega_b|_{X^I}\;=\;\{\omega|_{X^I}\,:\, \omega\in \Omega_b\}\,.
\end{equation}
For any pair $\omega,\nu\in \Omega_b|_{X^I}$, let
$\phi_{\omega,\nu}^I:X^I\rightarrow \mathbb C$
be given by
\[
\phi_{\omega,\nu}^I(t)\;=\;\langle \omega(t), \nu(t)\rangle\,,\quad t\in X^I\,.
\]
This map is continuous and bounded, and so $\phi_{\omega,\nu}^I$ extends to a unique continuous map
$\tilde\phi_{\omega,\nu}^I:\beta X^I\rightarrow\mathbb C$. By uniqueness of this continuous extension,
the form $\langle\cdot,\cdot\rangle_t^I$ on $\Omega_b|_{X^I}$ defined by
\[
\langle\omega,\nu \rangle_t^I\;=\;\tilde\phi_{\omega,\nu}^I(t)\,,\quad t\in\beta X^I\,,
\]
is a pre-inner product on $\Omega_b|_{X^I}$ for each $t\in\beta X^I$.
Let $H_t^I$ denote the Hilbert space completion of $\Omega_b|_{X^I}/\mathcal N_t^I$, where
\[
\mathcal N_t^I=\{\omega\in \Omega_b|_{X^I}\,:\,\tilde\phi_{\omega,\omega}^I(t)=0\}\;.
\]
If $\overline \omega^I(t)$ denotes the equivalence class of $\omega\in \Omega_b|_{X^I}$ in $H_t^I$, then
for $t\in X^I$ the map $\overline \omega^I(t)\mapsto \omega (t)$ is
well defined and is an isometric isomorphism from $\Omega_b|_{X^I}/\mathcal N_t^I$ onto $H_t$.
Thus, we shall identify $H_t^I=H_t$ for every $t\in X^I$ so that, under this identification, we have $\overline \omega^I(t)=\omega(t)$. Hence, for every $\omega\in\Omega_b$ we have a bounded vector field
\[
\overline\omega^I:\beta X^I \rightarrow \bigsqcup_{t\in \beta X^I}\,H_t^I
\]
We shall consider
\begin{equation}\label{defi ei}
\mathcal E^I\;=\;\{\overline\omega^I\,:\,\omega\in\Omega_b|_{X^I}\}\,,
\end{equation}
which is a vector space of bounded vector fields for which $t\mapsto\|\overline\omega^I(t)\|$ is continuous on $\beta X^I$.

\begin{definition}
Let $\Omega^I$ denote the set of all vector fields $\nu:\beta X^I \rightarrow \bigsqcup_{t\in \beta X^I}\,H_t^I$
with the property that for every $t_0\in\beta X^I$ and $\varepsilon>0$ there is an open set $U\subset\beta X^I$ containing
$t_0$ and a vector field $\overline\omega^I\in \mathcal E^I$ such that
$\|\nu(t)-\overline\omega^I(t)\|<\varepsilon$ for every $t\in U$.
\end{definition}

We shall say that each $\nu\in\Omega^I$, as defined above, is a \emph{local uniform limit} of vector fields in $\mathcal E^I$.

\begin{proposition}\label{prop: ext chb}
$(\beta X^I, \{H_t^I\}_{t\in \beta X^I},\Omega^{I})$ is a continuous Hilbert bundle.
\end{proposition}

\begin{proof}
Because each $\nu\in\Omega^I$ is a local uniform limit of vector fields in $\mathcal E^I$, axiom (III)
on the continuity of the map $\check{\nu}$
and axiom (IV) on the closure of $\Omega^I$ under local uniform limits are easily verified.

In order to prove axiom (I),
let $f\in C(\beta X^I)$ and let $\nu\in \Omega^I$,
and consider the bounded vector field $f\cdot\nu$ defined by $f\cdot\nu(t)=f(t)\nu(t)$, $t\in\beta X^I$. Assume $t_0\in\beta X^I$ and let $\varepsilon>0$
be given. By the continuity of $f$ and the definition of $\Omega^I$, there are
an open neighbourhood $U$ of $t_0$ in $\beta X^I$ and a $\overline\eta^I\in\mathcal E^I$
such that, for all $t\in U$,
$|f(t)-f(t_0)|<\epsilon/2$ and $\|\nu(t)-\overline \eta^I(t)\|< \epsilon/2$. Therefore,
\[
\|f\cdot\nu(t)-f(t_0)\, \overline \eta^I(t)\|\;<\;  \epsilon (\|\nu\|+\|f\|)\,,\quad\forall\,t\in U\,.
\]
Thus, $f\cdot \nu$ is a local uniform limit of vector fields in $\mathcal E^I$---hence, an element of $\Omega^I$.
This proves that $\Omega^I$ is a $C(\beta X^I)$-module under the pointwise action.

That leaves axiom (II).
However, in the presence of axioms (I), (III), and (IV), the axiom (II)
is equivalent to the axiom that $\{\nu(t)\,:\,\nu\in\Omega^I\}$ be dense in $H_{t}^I$, for each $t\in \beta X^I$ \cite{dixmier--douady1963}. This
seemingly weaker axiom is satisfied by $\Omega^I$ because
$\{\overline\omega^I(t)\,:\,\overline\omega^I\in\mathcal E^I\}$ is dense in $H_{t}^I$ for each $t\in \beta X^I$.
\end{proof}

\begin{definition} If $I$ is an ideal of $A$, we write
$A^I=A(\beta X^I, \{H_t^I\}_{t\in \beta X^I},\Omega^{I})$ for the spatial continuous trace C$^*$-algebra
associated with the continuous Hilbert bundle $(\beta X^I, \{H_t^I\}_{t\in \beta X^I},\Omega^{I})$
(see the last paragraph of section \ref{S:prelim}).
\end{definition}

\vskip  4pt
\noindent{\bf Notational Convention.} Assume that $U$ is an open subset of $T$ an let $f\in C(T)$. We shall write that $f\in C_0(U)$ 
whenever $f$ is an element  of the ideal $J=\{g\in C(T)\,:\,g(t)=0,\;\forall\,t\in T\setminus U\}$. Conversely, note that every $h\in C_0(U)$ extends to a continuous function $h:T\rightarrow \mathbb C$
by defining $h(t)=0$ for $t\in T\setminus U$. Thus, we shall sometimes consider $h$ as an element of $C_b(T)$.

\begin{lemma}\label{ess ideal mem} Let $I$ be an
essential ideal of $A$ and
suppose that $a\in A$. Then $a\in I$ if and only if $\check{a}\in C_0(X^I)$.
\end{lemma}

\begin{proof} For each $t\in T$ let $A_t=\{a(t)\,:\,a\in A\}$; by \cite[Theorem 4.4]{fell1961},
$A_t=K(H_t)$, the simple C$^*$-algebra of compact operators acting on $H_t$. Next,
let $I_t=\{b(t)\,:\,b\in I\}\subset A_t$. By \cite[Lemma 1.8]{fell1961}, if
$a\in A$, then  $a\in I$ if and only if $a\in I_t$ for all $t\in T$ . Because $I_t$ is an ideal of $A_t$,
we conclude that
$I_t=\{0\}$ for $t\in Z^I$
and $I_t=A_t$ for $t\in X^I$.
Hence, a necessary and sufficient condition for $a$ to belong to $I$ is that
$a(t)=0$ for all $t\in Z^I$. That is,
$a\in I$ if and only if $\check{a}\in C_0(X^I)$.
\end{proof}

\begin{proposition}\label{essen emb} There exists a monomorphism $\delta_I:I\rightarrow A^I$ such that
\begin{enumerate}
\item $\delta_I(I)$ is an essential ideal of $A^I$,
\item $\delta_I(a)\,(t)\,=\,a(t)$, for all $a\in I$ and $t\in X^I$, and
\item $\delta_I(a)\,(t)\,=\,0$, for all $a\in I$ and $t\in \beta X^I\setminus X^I$
\end{enumerate}
\end{proposition}

\begin{proof} The topological space $X^I$ is regarded now as an open dense subset of $\beta X^I$;
hence, $C_0(X^I)$ is an essential ideal of $C(\beta X^I)$.

For every $a\in I$, define an operator field $\mathfrak a: \beta X^I\rightarrow \bigsqcup_{t\in \beta X^I} K(H_t^I)$
by $\mathfrak a_{|X^I}=a_{|X^I}$ and $\mathfrak a(t)=0$ for all $t\in\beta X^I\setminus X^I$. We show below that $\mathfrak a\in A^I$.

By Lemma \ref{ess ideal mem}, $\check{a}\in C_0(X^I)$. Thus,
$\check{\mathfrak a}_{\vert X^I}\in C_0(X^I)$ and satisfies $\check{\mathfrak a}(t)=0$ for all $t\in \beta{X^I}\setminus X^I$.
Hence, $\check{\mathfrak a}\in C(\beta X^I)$.

To prove that $\mathfrak a$ is a weakly continuous operator field, it is sufficient to verify the weak continuity condition in vector fields in $\mathcal E^I$,
as every $\nu\in\Omega^I$ is a local uniform limit of vector fields in $\mathcal E^I$. To this end, let $\omega,\eta\in \Omega_b$ and
consider the function $h(t)=\langle \mathfrak a(t)\,\overline \omega^I(t),\overline \eta^I(t)\rangle$, $t\in\beta X^I$.
Restricted to $X^I$, $h$ is continuous (since $a\in A$) and vanishes at infinity. As noted earlier, the facts $h_{|X^I}\in C_0(X)$ and
$h(t)=0$ for  all $t\in \beta{X^I}\setminus X^I$ imply that $h\in C(\beta X^I)$. Thus, $\mathfrak a$ is a weakly continuous operator field.

Lastly, we show that $\mathfrak a$ is approximately finite-dimensional with respect to $\Omega^I$.
Notice that $a$ has this property (with respect to $\Omega_b|_{X^I}$) on $X^I$. Thus, at every point $t_0\in X^I$ and for every $\varepsilon>0$ there will
be an open neighbourhood $U$ of $t_0$ in $X^I$ such that $\mathfrak a$ is approximately finite-dimensional with respect to $\Omega^I$ to within $\varepsilon$ on $U$.
Assume now $t_0\in \beta X^I\setminus X^I$ and let $\varepsilon>0$. Since $\check{\mathfrak a}(t_0)=0$,
 there is an open set $U\subset\beta X^I$
containing $t_0$ such that $0\le\check{\mathfrak a}(t)<\varepsilon$ for all $t\in U$.
This shows that $\check{\mathfrak a}$ is approximately finite-dimensional with respect to $\Omega^I$ to within $\varepsilon$ on $U$.
This completes the proof that $\mathfrak a\in A^I$.

Now define $\delta_I:I\rightarrow A^I$ by $\delta_I(a)=\mathfrak a$. Clearly $\delta_I$ is a homomorphism. Because $a(s)=0$ for all
$s\in T\setminus X^I$, we have $\|a\|=\max_{t\in X^I}\|a(t)\|$. Thus,
\[
\|\delta_I(a)\|\;=\;\max_{t\in\beta X^I}\|\mathfrak a(t)\|\;=\;
\max_{t\in X^I}\|\mathfrak a(t)\|\;=\; \max_{t\in X^I}\| a(t)\|\;=\; \|a\|\,,
\]
which shows that $\delta_I$ is a monomorphism.

It remains to prove that $\delta_I(I)$ is an essential ideal of $A^I$.
Let
\[
\mathfrak I\;=\;\{y\in A^I\,|\,y(t)=0\,,\;\forall\,t\in\beta X^I\setminus X^I\}\,,
\]
which is an essential ideal of $A^I$ that contains $\delta_I(I)$. We claim that $\mathfrak I=\delta_I(I)$. Indeed, let $y\in\mathfrak I$.
Thus, $\check y\in C_0(X^I)$, where $X^I$ is viewed as an open dense subset of $\beta X^I$. Since $y\in A^I$ then $y|_{X^I}$ is an
operator field which is almost finite-dimensional with respect to $\Omega_b|_{X^I}$ on $X^I$. If we now consider $X^I$ as a
dense open subset of $T$ we conclude that
 $y_{|X^I}$ extends to an element $y_{\rm ext}\in A$ such that $y_{\rm ext}(s)=0$ for all $s\in T\setminus X^I$. Moreover, by
 Lemma \ref{ess ideal mem} $y_{\rm ext}\in I$. It is clear that in this case $\delta_I(y_{\rm ext})=y$ which proves the previous claim.
\end{proof}

%%%%%%%%%%%%%%%%%%%%%%%%%%%%%%%%%%%
\section{A Direct-Limit Continuous Hilbert Bundle}\label{S:hb}

For every essential ideal $I$ of $A$, we have constructed in the previous section a continuous Hilbert bundle
$(\beta X^I, \{H_t^I\}_{t\in \beta X^I},\Omega^{I})$ and considered
the
spatial continuous trace C$^*$-algebra $A^I$ associated with this Hilbert bundle. Our aim in this section is to use
these constructions to pass to limiting objects:
\[
\begin{array}{ll}
\Delta\;=\;\displaystyle\lim_{\leftarrow}\,\beta X^I                     & \mbox{(a compact, extremely disconnected Hausdorff space)} \\
C(\Delta) \;=\; \displaystyle\lim_\rightarrow\,C(\beta X^I)         & \mbox{(an abelian AW$^*$-algebra)} \\
H_s^\Delta\;=\;\displaystyle\lim_{\rightarrow}\,H_{\Phi_I(s)}^I & \mbox{(a Hilbert space, for every $s\in\Delta$)} \\
\Omega^\Delta\;=\;\displaystyle\lim_{\rightarrow}\,\Omega^I   & \mbox{(a Banach space of vector fields)} \\
(\Delta, \{H_s\}_{s\in\Delta}, \Omega^\Delta)                           & \mbox{(a continuous Hilbert bundle).}
\end{array}
\]

We recall here for the reader's convenience the notions of inverse system and inverse limit of a family of sets $\{X_\alpha\}_{\alpha\in\Lambda}$,
where $\Lambda$ is a directed set. Assume $\mathcal F$ is a family of functions indexed by
subsets of $\Lambda\times\Lambda$, whereby:
\begin{enumerate}
\item $f_{\alpha\alpha}={\rm id}_{X_\alpha}$;
\item if $(\alpha,\beta)$ satisfies $\alpha\le\beta$, then $f_{\alpha\,\beta}:X_\beta\rightarrow X_\alpha$;
\item if $\alpha\leq\beta\leq\gamma$, then $f_{\alpha\gamma}=f_{\alpha\beta}\circ f_{\beta\gamma}$.
\end{enumerate}
The triple $(\Lambda, \{X_\alpha\}_{\alpha\in\Lambda}, \mathcal F)$ is called an inverse system.

The inverse limit of the inverse system $(\Lambda, \{X_\alpha\}_{\alpha\in\Lambda}, \mathcal F)$
is the set denoted by $\displaystyle\lim_{\leftarrow}\, X_\alpha$ and defined to be the subset of the Cartesian
product $\displaystyle\prod_\alpha X_\alpha$ consisting of all $x=(x_\alpha)_\alpha$ with the
property that $x_\alpha=f_{\alpha,\beta}(x_\beta)$ whenever $\alpha\le\beta$. If $f_{\alpha}$ denotes the
projection of $\displaystyle\lim_{\leftarrow}\, X_\alpha$ onto $X_\alpha$, then
$f_\alpha=f_{\alpha\beta}\circ f_\beta$
whenever $\alpha\le\beta$. We shall make use of the fact that the projections $f_\alpha$ are surjective if
each $f_{\alpha\beta}$, $\alpha\leq\beta$, is surjective.

If each $X_\alpha$ is a topological space, if $\displaystyle\prod_\alpha X_\alpha$ has the product topology, and
if the functions in $\mathcal F$ are continuous, then the functions $f_\alpha:\displaystyle\lim_{\leftarrow}\, X_\alpha\rightarrow X_\alpha$
are continuous. When all $X_\alpha$ are compact, then so is $\displaystyle\lim_{\leftarrow}\, X_\alpha$.

The dual notions of direct system and direct limit are familiar to operator algebraists, and so they
will not be defined here.

 \begin{proposition}\label{systems1} {\rm (Inverse and Direct Systems)}
 \begin{enumerate}
 \item There exists an inverse system $( \mathcal I_{\rm ess}(A), \{\beta X^I\}_{I}, \{\Phi_{JI}\}_{J \preccurlyeq I})$
          of compact spaces and continuous surjections.
 \item There exists a direct system $( \mathcal I_{\rm ess}(A), \{\Omega^I\}_{I}, \{\lambda_{JI}\}_{J\preccurlyeq I})$
          of bounded vector fields and linear isometries.
 \end{enumerate}
 \end{proposition}

\begin{proof} By \cite[Theorem VII.7.3]{Dugundji-book}, the continuous embedding of a locally compact Hausdoff space $Y$ into its Stone--\v Cech
compactification $\beta Y$ is an open map. Hence, assuming that $J \preccurlyeq I$, we have that
$i_J(X^I)$ is open (and dense) in $\beta X^J$.
By the Universal Property of the Stone--\v Cech compactification, there is a (unique)
continuous $\Phi_{JI}:\beta X^I\rightarrow \beta X^J$ such that
\begin{equation}\label{used a lot}
\iota_J{}_{|X^I}\,=\,\Phi_{JI}\circ \iota_I\,.
\end{equation}
Moreover, $\Phi_{JI}$ is surjective
because the open set $i_J(X^I)$ is dense in $\beta X^J$.
Finally, it is evident that $\Phi_{II}={\rm id}_{X^I}$ and that $K \preccurlyeq J \preccurlyeq I$
leads to
$\Phi_{KI}=\Phi_{KJ}\circ\Phi_{JI}$. Hence, $( \mathcal I_{\rm ess}(A), \{\beta X^I\}_I, \{\Phi_{JI}\}_{J\preccurlyeq I})$ is an inverse system,
proving the first statement.

The second assertion requires an intermediate step that we shall use later on.
For every  $I,\,J\in  \mathcal I_{\rm ess}(A)$
for which  $J \preccurlyeq I$ and every $t\in \beta X^I$ we shall define a
unitary $\Psi_{JIt}$, such that
\begin{equation}\label{e:hs iso}
\left\{ \begin{array} {ll}
\Psi_{JIt}:H^J_{\Phi_{JI}(t)}\rightarrow H^I_t, & \mbox{for }\;  J \preccurlyeq I \\  & \\
\Psi_{KIt}=\Psi_{JIt}\circ\Psi_{KJ\Phi_{JI}(t)}, & \mbox{for }\; K \preccurlyeq J \preccurlyeq I
\end{array}
\right\}
\end{equation}
To achieve this we fix $t\in\beta X^I$ and $I\in  \mathcal I_{\rm ess}(A)$. Recall that for any $L\in \mathcal I_{\rm ess}(A)$ the
linear space $\{\overline\omega^L(s)\,:\,\omega\in \Omega_b{}_{|X^L})\}$ is dense in $H^L_s$.
Hence, if  $J \preccurlyeq I$,  the map $\overline\omega^J\left(\Phi_{JI}(t)\right)\mapsto \overline\omega^I(t)$
is a well defined linear isometry, and so it extends to a unitary
$\Psi_{JIt}:H^J_{\Phi_{JI}(t)}\rightarrow H^I_t$. Now if $K \preccurlyeq J \preccurlyeq I$, then $\Psi_{KIt}=\Psi_{JIt}\circ\Psi_{KJ\Phi_{JI}(t)}$
follows immediately from $\Phi_{KI}(t)=\Phi_{KJ}\circ\Phi_{JI}(t)$.

Now to prove our second assertion, assume
$I, J\in \mathcal I_{\rm ess}(A)$ are such that $J \preccurlyeq I$.
If $\nu\in\Omega^J$, then a vector field $\tilde\nu:\beta X^I\rightarrow\sqcup_{t\in\beta X^I}H_t^I$ is defined as follows:
\begin{equation}\label{heavy lambda def}
\tilde\nu\,(t)\;=\;\Psi_{JIt}\circ\nu\circ\Phi_{JI}\,(t)\,,\quad t\in\beta X^I\,.
\end{equation}
Observe that if $\nu\in\Omega_b$, then $\tilde{(\overline\omega^J)}=\overline\omega^I$.
Let $\lambda_{JI}$ be the function with domain $\Omega^J$
and defined by $\lambda_{IJ}\nu=\tilde\nu$. Note that $\lambda_{JI}$ is a linear transformation and that
\[
\sup_{t\in\beta X^I}\,\|\tilde\nu(t)\|\;=\; \sup_{t\in\beta X^I}\,\|\nu\left(\Phi_{JI}(t)\right)\|\;=\;\sup_{r\in\beta X^J}\,\|\nu(r)\|\,.
\]
The first equality above is on account of the operator $\Psi_{JIt}$ being an isometry, and the second is true because $\Phi_{JI}$ is a surjection.
Hence, $\tilde\nu$ is a bounded vector field of norm $\|\tilde\nu\|=\|\nu\|$. Because $\lambda_{JI}\left(\mathcal E^J\right)=\mathcal E^I$
and every $\nu\in\Omega^J$ is a local uniform limit of vectors fields in $\mathcal E^J$, we conclude that $\lambda_{JI}\nu$ is a local uniform limit of
vectors fields in $\mathcal E^I$, whence $\lambda_{JI}(\nu)\in\Omega^I$. Finally, by virtue of the properties
of $\Psi_{JIt}$ and $\Phi_{JI}$, we obtain $\lambda_{II}={\rm id}_{\Omega^I}$ and
$\lambda_{KI}=\lambda_{JI}\circ\lambda_{KJ}$ whenever $K \preccurlyeq J \preccurlyeq I$.
\end{proof}

\emph{Notation.} For the purposes of notational clarity, equation \eqref{heavy lambda def} is henceforth expressed more simply as
\begin{equation}\label{light lambda def}
\lambda_{JI}\nu\;=\;\nu\circ\Phi_{JI}\,.
\end{equation}
That is, \eqref{light lambda def} is shorthand for \eqref{heavy lambda def}.

Denote the inverse limit of the  inverse system $( \mathcal I_{\rm ess}(A), \{\beta X^I\}_I, \{\Phi_{JI}\}_{J\preccurlyeq I})$ by
\begin{equation}\label{delta def}
\Delta\;=\;\lim_{\leftarrow}\,\beta X^I\,,
\end{equation}
and let $\Phi_I:\Delta\rightarrow \beta X^I$ denote the continuous, surjective functions that
satisfy $\Phi_J=\Phi_{JI}\circ\Phi_I$ whenever $J\preccurlyeq I$. The space $\Delta$ is compact and Hausdorff. We shall note below that
$\Delta$ is also extremely disconnected;
thus, it is a \emph{Stonean} space.

If $J\preccurlyeq I$, then the continuous surjection $\Phi_{JI}:\beta X^I\rightarrow\beta X^J$ leads to a monomorphism $\rho_{JI}:C(\beta X^J)
\rightarrow C(\beta X^I)$ defined by
$\rho_{JI}(f)=f\circ\Phi_{JI}$ and in this way we produce a direct system of abelian C$^*$-algebras and monomorphisms.
By  \cite{semadeni1968},
\begin{equation}\label{semadeni}
C(\Delta) \;=\; \lim_\rightarrow\,C(\beta X^I)\,,
\end{equation}
the direct limit C$^*$-algebra of the system $( \mathcal I_{\rm ess}(A), \{C(\beta X^I)\}_{I}, \{\rho_{JI}\}_{J\preccurlyeq I})$.
Observe that \eqref{semadeni} states that
\[
\mloc\left( C_0(T) \right) \;=\; C(\Delta)\,.
\]
As the local multiplier algebra of an abelian C$^*$-algebra is an abelian AW$^*$-algebra \cite[Proposition 3.4.5]{Ara--Mathieu-book}, the
maximal ideal space of $\mloc\left( C_0(T) \right)$ is extremely disconnected, which is why $\Delta$ is Stonean.

Via the universal property, we deduce that the
algebraic direct limit of the system  $( \mathcal I_{\rm ess}(A), \{C(\beta X^I)\}_{I}, \{\rho_{JI}\}_{J\preccurlyeq I})$
is (identified with)
\begin{equation}\label{algebraic semadeni}
\displaystyle\mbox{\rm alg -}\lim_{\rightarrow}\,C(\beta X^I)\;=\;\{f\circ\Phi_I\,:\,I\in  \mathcal I_{\rm ess}(A),\;f\in C(\beta X^I)\}\,,
\end{equation}
which is uniformly dense in $C(\Delta)$.

We now construct Hilbert spaces $H_s^\Delta$. By \eqref{e:hs iso},
$( \mathcal I_{\rm ess}(A), \{H_{\Phi_I(s)}^I\}_{I}, \{\Psi_{JI\Phi_J(s)}\}_{J \preccurlyeq I})$
is a direct system of Hilbert spaces and unitaries, for each $s\in\Delta$.
Thus, we consider the Hilbert space direct limit
\begin{equation}\label{h delta}
H_s^\Delta\;=\;\lim_{\rightarrow}\,H_{\Phi_I(s)}^I\,
\end{equation}
(note that for $J\preccurlyeq I$, $H^J_{\Phi_J(s)}=H^I_{\Phi_I(s)}$).
Hence, for every $I\in  \mathcal I_{\rm ess}(A)$
there is a surjective linear isometry $\Psi_{I_s}:H_{\Phi_I(s)}^I\rightarrow H_s^\Delta$ such that
\[
\Psi_{J_s}\;=\;\Psi_{I_s}\circ \Psi_{JI\Phi_J(s)}\,,\quad\forall\, J \preccurlyeq I\,.
\]
Thus, the set
\begin{equation}\label{heavy algebraic h delta}
\{\Psi_{I_s}\nu\left(\Phi_I(s)\right)\,:\,I\in  \mathcal I_{\rm ess}(A),\;\nu\in\Omega^I\}
\end{equation}
is dense in $H_s^\Delta$. \emph{For notational simplicity}, we write \eqref{heavy algebraic h delta} as
\begin{equation}\label{algebraic h delta}
\{\nu\circ\Phi_I(s)\,:\,I\in  \mathcal I_{\rm ess}(A),\;\nu\in\Omega^I\}.
\end{equation}
Observe that the inner product in $H_s^\Delta$ of any two such vectors
$\nu_j\circ \Phi_{I_j}(s)$, $j=1,2$, is (well) defined by
\[
\left\langle \nu_1\circ \Phi_{I_1}(s),\,\nu_2\circ\Phi_{I_2}(s)\right\rangle \;=\;
\left\langle \nu_1\circ \Phi_{J}(s),\,\nu_2\circ\Phi_{J}(s)\right\rangle\,,
\]
for any $J\in \mathcal I_{\rm ess}(A)$ with $J \preccurlyeq I_1$ and  $J \preccurlyeq I_2$.

Likewise,
\begin{equation}\label{algebraic omega}
\displaystyle\mbox{\rm alg -}\lim_{\rightarrow}\,\Omega^I\;=\;\{\nu\circ\Phi_I\,:\,I\in  \mathcal I_{\rm ess}(A),\;\nu\in\Omega^I\}
\end{equation}
is an algebraic direct limit of vector spaces. Hence, every $\mu\in \displaystyle\mbox{\rm alg -}\lim_{\rightarrow}\,\Omega^I$ is a vector field $\Delta\rightarrow\bigsqcup_{s\in\Delta}H_s^\Delta$
via
\[
\mu(s)\;=\;\nu\left(\Phi_I(s)\right) \,\in\, H_s^\Delta\,,\;\mbox{for some }I\in \mathcal I_{\rm ess}(A)\mbox{ and } \nu\in\Omega^I\,.
\]

\noindent{\bf Notational Summary.} If $I,J\in \mathcal I_{\rm ess}(A)$  are such that $I\subset J$, and if $\nu\in\Omega^J$,
then
\begin{equation}\label{notation1}
\nu\circ\Phi_J\;=\;\nu'\circ\Phi_I\,,\;\mbox{ where }\;\nu'=\lambda_{JI}\nu=\nu\circ\Phi_{JI}\,.
\end{equation}

Our aim below is to complete $\displaystyle\mbox{\rm alg -}\lim_{\rightarrow}\,\Omega^I$ in a manner that will give it the
structure of a continuous Hilbert bundle over $\Delta$. Not only should this completion be closed under local uniform limits, but it should be
a $C(\Delta)$-module as well.

In what follows, let
\begin{equation}\label{algebraic omega E}
\mathcal E\;=\;\displaystyle\mbox{\rm alg -}\lim_{\rightarrow}\,\Omega^I\;=\;\{\nu\circ\Phi_I\,:\,I\in  \mathcal I_{\rm ess}(A),\;\nu\in\Omega^I\}\,.
\end{equation}

\begin{definition} $\Omega^\Delta$ is the set of all bounded vector fields
$\nu:\Delta\rightarrow\bigsqcup_{s\in\Delta}H^\Delta_s$ with the property that for
each $s_0\in \Delta$ and $\varepsilon >0$ there exist an open set $U\subset\Delta$
containing $s_0$ and $\omega\in\mathcal E$ such that $\|\nu(s)-\omega(s)\|<\varepsilon$
for all $s\in U$.
\end{definition}

\begin{proposition} $(\Delta, \{H_s^\Delta\}_{s\in\Delta}, \Omega^\Delta)$ is a continuous Hilbert bundle.
\end{proposition}

\begin{proof} Of the axioms to be satisfied, the only one that is not immediate
is axiom (I): that $\Omega^\Delta$ is a $C(\Delta)$-module. To prove this, let $\xi\in\Omega^\Delta$
and $f\in C(\Delta)$. Choose $s_0\in \Delta$ and $\varepsilon>0$. By the continuity of $f$, there is
an open neighbourhood $U_1\subset\Delta$ of $s_0$ such that $|f(s)-f(s_0)|<\frac{\varepsilon}{2\|\xi\|}$,
for all $s\in U_1$. By definition of $\Omega^\Delta$, there exist an open neighbourhood $U_2\subset \Delta$ of $s_0$,
an $I\in \mathcal I_{\rm ess}(A)$, and a $\nu\in\Omega^I$
such that $\|\xi(s)-\nu\circ\Phi_I(s)\|<\frac{\varepsilon}{2 \|f\| }$, for all $s\in U_2$. Let $U=U_1\cap U_2$ to obtain
\[
\|f\cdot\xi(s)-f(s_0)\nu\circ\Phi_I(s)\|\,<\,\varepsilon\,,\;\mbox{ for all }s\in U\,.
\]
Now as $f(s_0)\nu\circ\Phi_I\in\Omega^I$, the inequality above implies that $f\cdot \xi $ is a local uniform limit of
elements of $\displaystyle\mbox{\rm alg -}\lim_{\rightarrow}\,\Omega^I$. Hence, $f\cdot\xi\in\Omega^\Delta$.
\end{proof}

We call $(\Delta, \{H_s\}_{s\in\Delta}, \Omega^\Delta)$ the
\emph{direct limit continuous Hilbert bundle} of the system described in item (ii) of Proposition \ref{systems1}.

The next result shows that elements  of $\Omega^\Delta$ are not just local uniform limits of elements of
$\displaystyle\mbox{\rm alg -}\lim_{\rightarrow}\,\Omega^I$, but rather
each $\xi\in\Omega^\Delta$ is a {global} uniform limit of elements of
$\displaystyle\mbox{\rm alg -}\lim_{\rightarrow}\,\Omega^I$.

\begin{theorem}\label{omega limit}  $\Omega^\Delta\;=\;\displaystyle\lim_{\rightarrow}\,\Omega^I$
as a Banach space.
\end{theorem}

\begin{lemma}\label{name} Assume $\xi\in\Omega^\Delta$.
For every $s_0\in\Delta$ and $\varepsilon>0$ there exist $I\in \mathcal I_{\rm ess}(A)$,
$\nu\in\Omega^I$, and an open set $V\subset\beta X^I$ such that the open set $U=\Phi_I^{-1}(V)\subset\Delta$
contains $s_0$ and $\|\xi(s)-\nu\circ\Phi_I(s)\|<\varepsilon$ for all $s\in U$.
\end{lemma}

\begin{proof}
Let $s_0\in\Delta$. By definition, there are
$J\in \mathcal I_{\rm ess}(A)$, $\nu'\in\Omega^J$, and $U_1\subset\Delta$ such that $U_1$ is an open neighbourhood of $s_0$ and
$\|\xi(s)-\nu'\circ\Phi_J(s)\|<\varepsilon$ for all $s\in U_1$. Inside $U_1$ there is an open set $U$ containing $s_0$ such that $U$ has the
form $U=\Phi_{K}^{-1}(W)$, for some $K\in \mathcal I_{\rm ess}(A)$ and open set $W\subset \beta X^K$
\cite[Proposition 2.3 in Appendix Two]{Dugundji-book}. Consider the essential ideal $I=J\cap K$; thus,
$J\preccurlyeq I$ and $K\preccurlyeq I$, and so we consider the continuous functions
$\Phi_{JI}:\beta X^I\rightarrow\beta X^J$ and $\Phi_{KI}:\beta X^I\rightarrow\beta X^K$. Let $V=\Phi_{KI}^{-1}(W)\subset\beta X^I$
and $U=\Phi_I^{-1}(V)\subset\Delta$. The relation $\Phi_K=\Phi_{KI}\circ\Phi_I$ implies that
$\Phi_{K}^{-1}(W)=\Phi_I^{-1}\left( \Phi_{KI}^{-1}(W)\right)$. Thus, $s_0\in U=\Phi_K^{-1}(W)\subset U_1$. Now let $\nu\in\Omega^I$
be given by $\nu=\nu'\circ\Phi_{JI}$. Thus, for any $s\in U$, we have that
\[
\|\xi(s)-\nu\circ\Phi_I(s)\|\;=\;\|\xi(s)-\nu'\left(\Phi_{JI}\circ\Phi_I(s)\right)\|\;=\;\|\xi(s)-\nu'\circ\Phi_{J}(s)\|\;<\;\varepsilon\,,
\]
which completes the proof.
\end{proof}

We now prove Theorem \ref{omega limit}.

\begin{proof}
What we aim to prove is:
for each $\xi\in\Omega^\Delta$ and $\varepsilon>0$ there exist
$I\in \mathcal I_{\rm ess}(A)$ and $\nu\in\Omega^I$ such that $\|\xi(s)-\nu\circ\Phi_I(s)\|<\varepsilon$
for every $s\in \Delta$.

Fix $\varepsilon>0$. Lemma \ref{name} provides us with an open cover of $\Delta$ of a specific type.
Let $U_1,\dots, U_n$ be a finite subcover. Thus, for
each $1\leq i\leq n$ there are $I_i\in  \mathcal I_{\rm ess}(A)$,
$\nu_i\in\Omega^{I_i}$, and open sets $V_i\subset\beta X^{I_i}$ such that $U_i=\Phi_{I_i}^{-1}(V_i)\subset\Delta$
and $\|\xi(s)-\nu_i\circ\Phi_{I_i}(s)\|<\varepsilon$ for all $s\in U_i$.
Suppose that $\{\varphi_i\}_{i=1}^n$ is a partition of unity subordinate to the open cover
$\{U_i\}_{i=1}^n$.
By properties of the inverse limit \cite{semadeni1968}, for each $i$ there exists $J_i\in\mathcal I_{\rm ess}(A)$ and
$\psi_i\in C(\beta X^{J_i})$ such that $\|\varphi_i-\psi_i\circ\Phi_{J_i}\|<\frac{\varepsilon}{n}$.
Let $I=\bigcap_{i=1}^n  (I_i\cap J_i)\in\mathcal I_{\rm ess}(A)$ and let $X^I\subset T$ be the open
set corresponding to  $I$. Let $\psi_i'\in C(\beta X^{I})$ denote
$\psi_i'=\psi_i\circ\Phi_{J_iI}$; hence,
\[
\|\varphi_i-\psi_i'\circ\Phi_I\|\;=\;\|\varphi_i-\psi_i\circ\Phi_{J_iI}\circ\Phi_I\|\;=\;
\|\varphi_i-\psi_i\circ\Phi_{J_i}\|\;<\;\frac{\varepsilon}{n}\,.
\]
Consider now the following element $\nu\in \Omega^I$:
\[
\nu\;=\;\sum_{i=1}^n ( \psi_i'\circ\Phi_{I})\cdot \lambda_{I_iI}\nu_i\,.
\]
Then, for every $s\in\Delta$,
\[
\begin{array}{rcl}
\|\xi(s)-\nu\circ\Phi_I(s)\| &=& \left\|\displaystyle\sum_{i=1}^n\varphi_i(s)\xi(s)\,
-\,\displaystyle\sum_{i=1}^n\psi_i'(\Phi_{I}(s))\,\left(\nu_i\circ\Phi_{I_i}(s)\right)
\right\| \\ && \\
&\leq&
\displaystyle\sum_{i=1}^n\, \varphi_i(s)\,\left\| \xi(s)-\nu_i\circ\Phi_{I_i}(s)\right\|  \\ && \,+\,
\displaystyle\sum_{i=1}^n  \,|\varphi_i(s)-\psi_i'\circ\Phi_{I}(s)|\, \left\| \nu_i\circ\Phi_{I_i}(s) \right\|  \\ && \\
  & < & \varepsilon + \varepsilon ( \| \xi(s)  \|  + \varepsilon )  \,.
\end{array}
\]
Hence, $\|\xi-\nu\circ\Phi_I\|< \varepsilon + \varepsilon ( \| \xi  \|  + \varepsilon )$.
\end{proof}

%%%%%%%%%%%%%%%%%%%%%%%%%%%%%%%%%%%%%%%
\section{A Direct Limit C$^*$-Algebra}\label{S:limit algs}

In this section we keep the notation from the previous sections.
In particular, we use the maps $\delta_I$ from Proposition \ref{essen emb}.

\begin{proposition}\label{systems2}
There exists a direct system $( \mathcal I_{\rm ess}(A), \{A^I\}_{I}, \{\pi_{JI}\}_{J \preccurlyeq I})$
of C$^*$-algebras and monomorphisms
such that, for all $J \preccurlyeq I$,
$\delta_I=\pi_{JI}\circ\delta_J{}_{\vert I}$.
\end{proposition}

\begin{proof}
Assume that $J\preccurlyeq I$. For each $a\in A^J$, consider the bounded cross section $\tilde a$ of the fibred space
$(\beta X^I, \{B(H^I_t)\}_{t\in\beta X^I})$ that is defined by
\begin{equation}\label{heavy pi def}
\tilde a(t)\;=\;[\Psi_{JIt}]\,[a\left(\Phi_{JI}\,(t)\right)] \, [\Psi_{JIt}]^{-1} \,,\quad t\in\beta X^I\,.
\end{equation}
As before, we simplify the notation so that
\begin{equation}\label{light pi def}
\tilde a\;=\; a \circ \Phi_{JI}
\end{equation}
is now a shorthand expression of \eqref{heavy pi def}.

Let us now show that $\tilde a\in A^I$.  Continuity of $\check{\tilde a}$ follows from $\check{\tilde a}=\check{a}\circ\Phi_{JI}$. To show that $\tilde a$ is weakly continuous,
it is sufficient to use vector fields from $\mathcal E^I$. To this end, let $\omega_1,\omega_2\in\Omega_b$. Then
\[
\langle \tilde a(t)\,\overline\omega_1^I(t), \,\overline\omega_2^I(t)\rangle \;=\;
\left\langle   a(\Phi_{JI}(t))\,\overline\omega_1^J(\Phi_{JI}(t)), \,\overline\omega_2^J(\Phi_{JI}(t))\right\rangle\,,
\]
which is continuous as a function of $t\in\beta X^I$.

To show that $\tilde a$ is approximately finite-dimensional, select $t_0\in\beta X^I$ and $\varepsilon>0$. Consider $r_0=\Phi_{JI}(t_0)\in\beta X^J$.
Because $a\in A^J$, there is an open set $V\subset\beta X^J$ and $\nu_1,\dots,\nu_n\in\Omega^J$ such that, for every $r\in U$,
$\mbox{Span}\,\{\nu_1(r),\dots,\nu_n(r)\}$ is $n$-dimensional and $\|a(r)-p_ra(r)p_r\|<\varepsilon$, where $p_r\in B(H_r^J)$ is the projection
with range $\mbox{Span}\,\{\nu_1(r),\dots,\nu_n(r)\}$. Let $U=\Phi_{JI}^{-1}(V)$, an open neighbourhood of $t_0$. Consider the
vector fields $\lambda_{JI}\nu_\ell\in\Omega^I$, $1\leq\ell\leq n$. Because $\lambda_{JI}\nu_\ell(t)=\Psi_{JIt}\left(\nu(\Phi_{JI}(t))\right)$,
the subspace $\mbox{Span}\,\{\lambda_{JI}\nu_1(t),\dots,\lambda_{JI}\nu_n(t)\}$ is $n$-dimensional for all $t\in U$.
Let $q_t\in B(H_t^I)$ denote the projection onto this subspace, for each $t\in U$.
Then  $q_t=\Psi_{JIt}p_{\Phi_{JI}(t)}\Psi_{JIt}^{-1}$, which yields
$\|\tilde a(t)-q_t\tilde a(t) q_t\|<\varepsilon$ for all $t\in U$.

Define $\pi_{JI}:A^J\rightarrow A^I$ by $\pi_{JI}(a)=a\circ\Phi_{JI}$. It is now straightforward to verify that $\pi_{JI}$ is a homomorphism,
that $\pi_{JI}$ is isometric (since $\Phi_{JI}$ is surjective),
and that  $( \mathcal I_{\rm ess}(A), \{A^I\}_{I}, \{\pi_{JI}\}_{J \preccurlyeq I})$ is a direct system
of C$^*$-algebras and monomorphisms.

To prove that $\delta_I=\pi_{JI}\circ\delta_J{}_{\vert I}$, assume that $a\in I$. Thus, $\delta_I(a)$ is an operator field on $\beta X^I$
that vanishes on $\beta X^I\setminus X^I$ and agrees with $a$ on $X^I$. Thinking now of $I$ sitting inside $J$,
$\delta_J(a)$ is an operator field on $\beta X^J$ that vanishes on $\beta X^J\setminus X^J$.
Therefore the operator field $\pi_{JI}\delta_J(a)$ on $\beta X^I$ vanishes on $\beta X^I\setminus X^I$
because
$\Phi_{JI}$ maps $\beta X^I\setminus X^I$ into $\beta X^J\setminus X^J$ \cite[Theorem 6.12]{Gillman-Jerison-book}.
Hence, $\pi_{JI}\circ\delta_J(a)\in\delta_I(I)$. It is now straightforward to verify that  $\delta_I=\pi_{JI}\circ\delta_J{}_{\vert I}$.
\end{proof}

\noindent{\bf Notational Summary.} If $I,J\in \mathcal I_{\rm ess}(A)$  are such that  $J\preccurlyeq I$ and if $a\in A^J$,
then
\begin{equation}\label{notation2}
a\circ\Phi_J\;=\;a'\circ\Phi_I\,,\;\mbox{ where }\;a'=\pi_{JI}(a)=a\circ\Phi_{JI}\,.
\end{equation}

 Therefore, if $a\in A^I$ we have $a\circ \Phi_I:\Delta\rightarrow \bigsqcup_{s\in\Delta}B(H_s^\Delta)$, which induces a C$^*$-embedding of $A^I$ into  $\bigsqcup_{s\in\Delta}B(H_s^\Delta)$. Moreover, these embeddings are compatible with the direct system structure of $(\mathcal I_{\rm ess}(A), \{A^I\}_{I}, \{\pi_{JI}\}_{J \preccurlyeq I})$ of Proposition \ref{systems2}. Therefore
 if $A_\Delta$ denotes the norm-closure  $$A_\Delta:=\left(\displaystyle\bigcup_{I\in \mathcal I_{\rm ess}(A)}\{a\circ\Phi_I\,:\,a\in A^I\}\right)^{-\ \|\cdot\|}$$ then $A_\Delta=\lim_{\rightarrow}\,A^I$; that is, $A_\Delta$ is a concrete realisation of a C$^*$-limit of the directed system $\{A^I\}_{I\in \mathcal I_{\rm ess}(A)}$.

Proposition \ref{they are equal} below identifies the algebras $A^\Delta$ and $K(\Omega^\Delta)$, which were studied as separate
entities in the prequel \cite{argerami--farenick--massey2010}.
 
\begin{proposition}\label{they are equal}  Let $A^\Delta=A\left(\Delta, \{H_s\}_{s\in\Delta}, \Omega^\Delta\right)$ be
the continuous trace C$^*$-algebra associated with the continuous Hilbert bundle
$(\Delta, \{H_s\}_{s\in\Delta}, \Omega^\Delta)$. Then
\[
K(\Omega^\Delta)\,=\,A_\Delta \,= \,A^\Delta\,.
\]
\end{proposition}

\begin{proof} We first show that $\{a\circ\Phi_I\,:\,a\in A^I\}\subset A^\Delta$, for every $I\in \mathcal I_{\rm ess}(A)$. Suppose that $a\circ\Phi_I$ and that $\omega_1,\omega_2\in\Omega^\Delta$ are of the form $\omega_i=\omega_i'\circ\Phi_{J_i}$
for some $J_1,J_2\in \mathcal I_{\rm ess}(A)$ and $\omega_i'\in\Omega^{J_i}$. Let $K=I\cap J_1\cap J_2$, an essential ideal of $A$ such that $I\preccurlyeq K$ and $J_i\preccurlyeq K$.
Because $a\circ\Phi_I=(a\circ\Phi_{IK})\circ\Phi_K$ and $\omega_i'\circ\Phi_{J_i}=(\omega_i'\circ\Phi_{J_iK})\circ\Phi_K$, the continuity of the map
$s\mapsto\langle a\left(\Phi_I(s)\right)\omega_1(s),\omega_2(s)\rangle$ is immediate. As vector fields of the form $\omega=\omega'\circ\Phi_{J}$ are
uniformly dense in $\Omega^\Delta$, the operator field $a\circ\Phi_I$ is weakly continuous.

To show that $a\circ\Phi_I$ is almost finite-dimensional, assume $s_0\in \Delta$ and $\varepsilon>0$.
Let $t_0=\Phi_I(s_0)$. As $a$ is almost finite-dimensional, there
are an open set $V\in\beta X^I$ containing $t_0$ and $\omega_j\in\Omega^I$, $1\leq j\leq n$, such that, for all $t\in V$,
$\{\omega_j(t)\}_{j=1}^n$
is a set of linearly independent vectors and $\|a'(t)-p_ta'(t)p_t\|<\varepsilon$, where $p_t$ is the projection onto the span of
$\{\omega_j(t)\}_{j=1}^n$. Pull back to $\Delta$ using the open neighbourhood $U=\Phi_I^{-1}(V)$ of $s_0$
and rank-$n$ projections $q_s=p_{\Phi_I(s)}$ onto the span $\{\omega_j(\Phi_I(s))\}_{j=1}^n$ to obtain $\|a(s)-q_sa(s)q_s\|<\varepsilon$
for all $s\in U$. This completes the proof that $\{a\circ\Phi_I\,:\,a\in A^I\}\subset A^\Delta$.

In order to prove the equalities above, note that by the first part of the proof
we obtain that $A_\Delta\subset A^\Delta$.

By Theorem \ref{omega limit} any $\xi\in \Omega^\Delta$ is uniformly approximated to within $\varepsilon$
on $\Delta$ by some $\omega=\nu\circ\Phi_J$ of norm within $\varepsilon$ of $\|\xi\|$; therefore, we can conclude that
\[
\|\Theta_{\xi,\xi}\,-\,\Theta_{\omega,\omega}\|\;\leq\;   \|\check{\xi}-\check{\omega}\|(\|\check{\xi}\|+\|\check{\omega}\|)\;\leq\; C\varepsilon\,,
\]
where $C$ is a constant depending on $\|\xi\|$. As the set of all finite sums of the form $\Theta_{\xi,\xi}$ is dense in the positive cone of $K(\Omega^\Delta)$
\cite[Lemma 4.2]{argerami--farenick--massey2010}, we deduce that
$K(\Omega^\Delta)\subset A_\Delta$.

To conclude, we now prove that $A^\Delta\subset K(\Omega^\Delta)$.
Select $a\in A^\Delta$ and $\varepsilon>0$. For every $s_0\in\Delta$ there are an open set $U_{s_0}\subset\Delta$
containing $s_0$ and vector fields $\omega_1,\ldots,\omega_n\in \Omega^\Delta$ such that $\|a(s)-p_s\,a(s)\,p_s\|< \varepsilon$ $s\in U_{s_0}$,
where $p_s$ is the orthogonal projection onto Span$\{\omega_j(s):\ 1\leq j\leq n\}$. It turns out that $h_{s_0}:=p_s\,a(s)\,p_s\in \mathcal F(\Omega^\Delta)$
(see the proof of Lemma \ref{C=A}).
Therefore, $\{U_{s_0}\}_{s_0\in\Delta}$ is an open cover of $\Delta$ from which a finite subcovering $U_1,\dots, U_n$
exists; let $h_j\in  \mathcal F(\Omega^\Delta)$ denote the local approximant of $a$ on $U_j$, for each
$1\leq j\leq n$, and let $\{\varphi_j\}_{1\leq j\leq n}\subset C(\Delta)$ be a partition of unity subordinate to $\{U_j\}_{1\leq j\leq n}$.
Because $\cF(\Omega^\Delta)$ is a $C(\Delta)$-module
we have $h=\sum_{j=1}^n\varphi_i\cdot h_j\in
\mathcal F(\Omega^\Delta)$. Therefore, for every $s\in\Delta$,
\[
\|a(s)-h(s)\|\;\leq\;\sum_{j=1}^n\varphi_i(s)\,\|a(s)-h_j(s)\|\;<\;\varepsilon\,.
\]
Hence, $\|a-h\|<\varepsilon$ and so $a\in K(\Omega^\Delta)$.
\end{proof}

\section{A chain of inclusions of C$^*$-algebras}\label{sect:chain of inclusions}

An alternate description of $A$ is useful; we adopt terminology used in
\cite{akemann--pedersen--tomiyama1973}.
Consider $\Omega_0$ as a Hilbert C$^*$-module over $C_0(T)$.

Every $\kappa\in\mathcal F(\Omega_0)$ is a cross section of the fibred space $(T,\{ K(H_t)\}_{t\in T})$, and the set
$\mathcal F(\Omega_0)$ has the following properties: (i)
$\mathcal F(\Omega_0)$ is a $*$-algebra with respect to pointwise operations; (ii) $\{\kappa(t)\,:\,\kappa\in \mathcal F(\Omega_0)\}$
is dense in $K(H_t)$ for all $t\in T$; and (iii) $\check{k}\in C_0(T)$, for each $\kappa\in \mathcal F(\Omega_0)$.

A cross section $a$ of the fibred space $(T,\{ K(H_t)\}_{t\in T})$ is said to be \emph{continuous with respect to $\cF(\Omega_0)$} if
for each $t_0\in T$ and $\varepsilon>0$ there exist $\kappa\in \cF(\Omega_0)$ and an open set $U\subset T$ containing $t_0$
such that $\|a(t)-\kappa(t)\|< \varepsilon$ for every $t\in U$ .
(The terms ``continuous with respect to'' and ``local uniform limit of'' have the same meaning; however, as
the former terminology is used in the paper \cite{akemann--pedersen--tomiyama1973}, we adopt this
phrase here.)

Let $C=C_0(T,\{K(H_t)\}_{t\in T},\cF(\Omega_0))$ be the set of all cross sections $a$
of the fibred space $(T,\{ K(H_t)\}_{t\in T})$ that are continuous with respect to $\cF(\Omega_0)$ and satisfy $\check{a}\in C_0(T)$.
With respect to pointwise operations and the supremum norm, $C$ is a C$^*$-algebra.

\begin{lemma}\label{C=A}
$A(T,\{H_t\}_{t\in T},\Omega)\,=\,C_0(T,\{K(H_t)\}_{t\in T},\cF(\Omega_0))$.
\end{lemma}

\begin{proof} By construction,
$\cF(\Omega_0)\subset A$. Therefore, since $A$ is closed under local uniform approximation,  $C\subset A$.
Conversely, assume $a\in A$. Let $t\in T$ and $\varepsilon>0$. Thus, there exist an open set $V\subset T$ containing $t$
and $\omega_i\in \Omega$, $1\leq i\leq n$, such that, for every $s\in V$, the set of vectors
 $\{\omega_i(s)\}_{i=1}^n$ is a linearly independent set
and $\|a(s)-p_s \, a(s)\, p_s\|<\varepsilon$, where $p_s\in B(H_s)$ denotes the
projection onto $\mbox{Span}\,\{\omega_i(s):\, 1\leq i\leq n\}$.
Via the Gram--Schmidt process \cite[Lemma 4.2]{fell1961}, we may assume that the vectors  $\omega_i(s)$, $1\leq i\leq n$,
are pairwise orthogonal for every $s$ in some open set $U\subset V$ containing $t$. By Urysohn's Lemma,
we can also assume that each $\omega_i\in \Omega_0$.

For each $1\leq i,j\leq n$, let $f_{ij}(t)=\langle a(t)\omega_i(t),\omega_j(t)\rangle$, $t\in T$.
Thus,  $f_{ij}\in C_0(T)$ and so $f_{ij}\cdot\omega_i\in \Omega_0$ for all $i,j$. Now note that
\[
p_s \,a(s)\, p_s\;=\;\sum_{i,\,j=1}^n \langle a(s)\,\omega_j(s),\omega_i(s)\rangle\ \Theta_{\omega_i,\,\omega_j}(s)
\;\in\, \mathcal F(\Omega_0)\,,\quad\forall\,s\in U\,.
\]
Hence, $a$ is continuous with respect to $\mathcal F(\Omega_0)$, which proves that $A\subset C$.
\end{proof}

The previous result implies the following convenient description of the multiplier algebras of essential ideals.
If $I$ is an ideal of $A$, then  by Lemma \ref{C=A}, the spatial continuous trace C$^*$-algebra $A^I$ is given
by $A^I=C_0(\beta X^I,\{K(H)_t^I)\}_{t\in \beta X^I},\cF(\Omega^I))$. In viewing $A^I$ in this way,
the ideal $\delta_I(I)$ is given by
\[
I\;\cong\; \delta_I(I)\;=\;C_0 \left(X^I,\{K(H^I_t)\}_{t\in X^I},\cF((\Omega^I{}_{\vert X^I})_0) \right)\,.
\]
In this framework,
$x\in M(I)$ if and only if $x$ is a bounded cross
section of the fibred space $(X^I,\{B(H^I_t)\}_{t\in X^I})$ for which $x$
is \emph{strictly continuous} with respect to $\cF((\Omega^I{}_{\vert X^I})_0)$
\cite[Theorem 3.3]{akemann--pedersen--tomiyama1973}.
That is, for each $t_0\in X^I$, $a\in \cF((\Omega^I{}_{\vert X^I})_0)$, and $\varepsilon>0$ there are an open set
$U\subset X^I$ containing $t_0$ and $b\in \cF((\Omega^I{}_{\vert X^I})_0)$ such that
\[
\|\left( x(t)-b(t)\right)a(t)\|\,+\,\|a(t)\left( x(t)-b(t)\right)\|\;<\;\varepsilon\,,\;\mbox{ for all }\;t\in U\,.
\]
We summarise this fact in the next proposition.

\begin{proposition}\label{mult alg} If $I$ is an essential ideal of $A$, then
$x\in M(I)$ if and only if
$x$ is a bounded cross
section of the fibred space $(X^I,\{B(H^I_t)\}_{t\in X^I})$ such that $x$
is strictly continuous with respect to $\cF((\Omega^I{}_{\vert X^I})_0)$.
\end{proposition}

\begin{theorem}\label{k-omega} There exists a monomorphism $\gamma: K(\Omega^\Delta)\rightarrow \mloc(A)$
\end{theorem}

\begin{proof} We shall exploit the fact that $A_\Delta=K(\Omega^\Delta)$ (Proposition \ref{they are equal}).
Fix
$J\in \mathcal I_{\rm ess}(A)$ and let
$\gamma_J:A^J\rightarrow M(J)$ be the canonical embedding of $A^J$ into $M(J)$,
using the fact that $J\cong\delta_K(J)$
is an essential ideal of $A^J$. Recall that, by Lemma \ref{C=A}, that
\[
J\;\cong\; \delta_J(J)\;=\;C_0 \left(X^J,\{K(H^J_t)\}_{t\in X^J},\cF((\Omega^J{}_{\vert X^J})_0) \right)
\]
and, by Proposition \ref{mult alg},
$x\in M(J)$ if and only if $x$ is a bounded cross
section of the fibred space $(X^J,\{B(H^J_t)\}_{t\in X^J})$ which
is strictly continuous  with respect to $\cF((\Omega^J{}_{\vert X^J})_0)$.

Suppose now that $J \preccurlyeq I$ and let $x\in M(J)$. Because $\Phi_{JI}\circ \iota_I=\iota_J{}_{|X^I}$, $x\circ\Phi_{JI}{}_{\vert X^I}$
is a well defined bounded section, which we denote by $\tilde x$, of the fibred space $(X^I,\{B(H^I_t)\}_{t\in X^I})$.
Select $t_0\in X^I$ and $\varepsilon>0$.
Let $s_0\in\Phi_{JI}\in X^J$. Since $x\in M(J)$,
there are an open set $V\subset X^J$ containing $s_0$ and $a,b\in \cF((\Omega^J{}_{\vert X^J})_0) $ such that
\[
\|\left( x(s)-b(s)\right)a(s)\|\,+\,\|a(s)\left( x(s)-b(s)\right)\|\;<\;\varepsilon\,,\;\mbox{ for all }\;s\in V\,.
\]
Let $U=\Phi_{JI}^{-1}(V)$ and observe that $\pi_{JI}(a),\pi_{JI}(b)\in \cF((\Omega^IJ{}_{\vert X^I})_0)$. Hence, the
pull back to $X^I$ of the inequality above holds for $\tilde x$ in $U$ and, thus, $\tilde x\in M(I)$.

Define $\tilde\pi_{JI}:M(J)\rightarrow M(I)$ by $\pi_{JI}(x)=x\circ\Phi_{JI}{}_{\vert X^I}$. Thus $\tilde\pi_{JI}$ is a homomorphism
and satisfies the commutative diagram
 \begin{equation}\label{compat1}
\begin{CD}
A^J  @>{\pi_{JI}}>> A^I  \\
@V{\gamma_J}VV           @VV{\gamma_I}V  \\
M(J)  @>>{\tilde\pi_{JI}}> M(I)
\end{CD}\,.
\end{equation}
If $\tilde\pi_{JI}(x)=0$, then $x(s)=0$ for all $s\in \Phi_{JI}(X^I)=X^J$, whence $x=0$. Therefore, $\tilde\pi_{JI}$ is a monomorphism. Hence,
$( \mathcal I_{\rm ess}(A), \{A^I\}_{I}, \{\pi_{JI}\}_{J \preccurlyeq I})$ is a subsystem of  $( \mathcal I_{\rm ess}(A), \{M(I)\}_{I}, \{\tilde\pi_{JI}\}_{J \preccurlyeq I})$. Let
$N$ denote the direct limit of the
direct system $( \mathcal I_{\rm ess}(A), \{M(I)\}_{I}, \{\tilde\pi_{JI}\}_{J \preccurlyeq I})$.
The previous facts imply that there is a
monomorphism $\gamma:A_\Delta\rightarrow N$ such that
\[
\begin{CD}
A^I  @>{\pi_{I}}>> A_\Delta \\
@V{\gamma_I}VV           @VV{\gamma}V  \\
M(I)  @>>{\tilde\pi_{I}}> N
\end{CD}
\]
is a commutative diagram for all $I\in \mathcal I_{\rm ess}(A)$, where $\pi_I$ and $\tilde\pi_I$ are the embeddings of $A^I$
and $M(I)$ into their respective direct limits which satisfy $\pi_J=\pi_I\circ\pi_{JI}$ and $\tilde\pi_J=\tilde\pi_I\circ\tilde\pi_{JI}$ for all $J \preccurlyeq I$.

On the other hand, since $\delta_I=\pi_{JI}\circ\delta_J{}_{\vert I}$ (Proposition \ref{systems2}),
the commutativity of the previous diagram implies that $\tilde \pi_{JI}$ is the unique monomorphism induced
by the inclusion of essential ideals $\delta_J(I)\subset \delta_J(J)$, by the Universal Property of Multiplier Algebras. Therefore, $N$
and $\mloc(A)$ are canonically isomorphic and, thus, we may identify them.
\end{proof}

\begin{theorem}\label{a-komega-mloc inclusion}
There exists a monomorphism $\beta:A\rightarrow K(\Omega^\Delta)$ such that
\[
\begin{CD}
A @>\beta >> K(\Omega^\Delta)
@>\gamma>> \mloc(A)
\end{CD}
\]
is the canonical embedding of $A$ into its local multiplier algebra.
\end{theorem}

\begin{proof} Let $j:A\rightarrow M(A)$ denote the canonical embedding of $A$ into $M(A)$. Because
$\delta_A$ embeds $A$ as an essential ideal of $A^A$, the Universal Property of Multiplier Algebras
tells us that the homomorphism $\gamma_A:A^A\rightarrow M(A)$
in Theorem \ref{k-omega} is the unique embedding for which
\[
\begin{CD}
A   @>{\delta_A}>> A^A \\
@V{j}VV           @VV{\gamma_A}V  \\
M(A)  @ >>{\rm id}> M(A)
\end{CD}
\]
is a commutative diagram. Therefore,
by Proposition \ref{they are equal} and Theorem \ref{k-omega}, the diagram
\[
\begin{CD}
A @>\delta_A >> A^A @>\pi_A >> A_\Delta=K(\Omega^\Delta) \\
@V{j}VV           @V{\gamma_A}VV @VV{\gamma}V \\
M(A)  @ >>{\rm id}>M(A)@>>{\tilde\pi_A}> \mloc(A)\,.
\end{CD}
\]
is commutative. Thus,
\[
\begin{CD}
A @>\delta_A >> A^A
@>\gamma_A >> M(A)
@>\tilde\pi_A>> \mloc(A)
\end{CD}
\]
is a canonical embedding of $A$ into $\mloc(A)$. Therefore, if $\beta=\pi_A\circ\delta_A$, then
\[
\begin{CD}
A @>\beta >> K(\Omega^\Delta)
@>\gamma>> \mloc(A)
\end{CD}
\]
is also a canonical embedding of $A$ into $\mloc(A)$.
\end{proof}
%%%%%%%%%%%%%%%%%%%%%%%%%%%%%%%%%%%%%%%%%
\section{Main Results}\label{S:main results}

%%%%%%%
\subsection{Determination of the injective envelope}

To this point our analysis has made extensive use of continuous Hilbert bundles
for the study of $A$ and its essential ideals,
but for the determination of the injective envelope and local multiplier algebras
of $A$, a larger class of vector fields is required. We shall now draw upon our
work in the prequel \cite{argerami--farenick--massey2010}  to the present paper.

\begin{definition}\label{definition:weakly continuous}  
A vector field $\mu:\Delta\rightarrow\bigsqcup_{s\in\Delta},H_s^\Delta$ is said to be \emph{weakly continuous} with
respect to the continuous Hilbert bundle $(\Delta,\{H_s^\Delta\}_{s\in\Delta},\Omega^\Delta)$ if the function
\[
s\mapsto \langle \mu(s),\xi(s)\rangle
\]
is continuous for all $\xi\in\Omega^\Delta$.

If $\Omega_{\rm wk}^\Delta$
is the vector space of all weakly continuous vector fields with respect
to the bundle $(\Delta,\{H_s^\Delta\}_{s\in\Delta},\Omega^\Delta)$, then the quadruple
$(\Delta,\{H_s^\Delta\}_{s\in\Delta},\Omega^\Delta, \Omega_{\rm wk}^\Delta)$ is
called a \emph{weakly continuous Hilbert bundle}.
\end{definition}

\begin{definition} {\rm (\cite{kaplansky1953})}
A Hilbert C$^*$-module $E$ over an abelian AW$^*$-algebra $Z$ is called
a {\em Kaplansky--Hilbert module} if the following three
properties hold:
\begin{enumerate}
  \item\label{definition:Kaplansky-Hilbert module:1} if $c_i\cdot \nu=0$ for some family $\{c_i\}_i\subset Z$ of pairwise-orthogonal
                projections and $\nu\in E$, then also $c\cdot\nu=0$,
               where $c=\sup_i\,c_i$;
  \item\label{definition:Kaplansky-Hilbert module:2} if
$\{c_i\}_i\subset Z$ is a family of pairwise-orthogonal
projections such that $1=\sup_i\,c_i$, and if $\{\nu_i\}_i\subset
E$ is a bounded family, then there is a $\nu\in E$
such that $c_i\cdot\nu=c_i\cdot\nu_i$ for all $i$;
  \item if $\nu\in E$, then $g\cdot\nu=0$ for all $g\in Z$ only if $\nu=0$.
\end{enumerate}
\end{definition}

The element $\nu\in E$
described in (ii) will be denoted by
\begin{equation}\label{aw sum}
\nu\;=\;\sum_ic_i\cdot\nu_i\,.
\end{equation}

\begin{theorem}\label{wchb} {\rm (\cite{argerami--farenick--massey2010})} The vector space
$\Omega_{\rm wk}^\Delta$ is a Kaplansky--Hilbert module over the abelian AW$^*$-algebra $C(\Delta)$, where the
C$(\Delta)$-valued inner product $\langle\cdot,\cdot\rangle$ on  $\Omega_{\rm wk}^\Delta$ has the property that for every pair $\xi,\eta\in\Omega_{\rm wk}^\Delta$
there is a meagre subset $M_{\xi,\eta}\subset\Delta$ such that
\[
\langle\xi,\eta\rangle\,(s)\;=\;\langle\xi(s),\eta(s)\rangle\,,\quad\mbox{for all }s\in\Delta\setminus M_{\xi,\eta}\,.
\]
\end{theorem}

Kaplansky \cite{kaplansky1953} proved that the C$^*$-algebra of bounded adjointable endomorphisms
of a Kaplansky--Hilbert module  is an
AW$^*$-algebra of type I and Hamana \cite{hamana1981} proved that every type I AW$^*$-algebra is injective. Thus:

\begin{corollary}
The C$^*$-algebra of $B(\Omega_{\rm wk}^\Delta)$ of all bounded adjointable endomorphisms of $\Omega_{\rm wk}^\Delta$ is an injective
AW$^*$-algebra of type I.
\end{corollary}

To determine the injective envelope of $A$ we use the following criterion.
Recall that an embedding or inclusion of a C$^*$-algebra $B$ into an injective C$^*$-algebra $C$ is said be \emph{rigid} if the only
unital completely positive linear map $\varphi:C\rightarrow C$ that is the identity on $B$ is the map $\varphi={\rm id}_C$. In \cite{hamana1979b} Hamana shows that
a necessary and sufficient condition for an injective C$^*$-algebra $C$ to be an
injective envelope of one its C$^*$-subalgebras $B$ is that the inclusion $B\subset C$ be rigid.

\begin{theorem}\label{ie stonean} {\rm (\cite{argerami--farenick--massey2010})}
There exists a monomorphism $\alpha:A^\Delta\rightarrow B(\Omega_{\rm wk}^\Delta)$
such that:
\begin{enumerate}
\item $\alpha(a)\nu\,(s)=a(s)\nu(s)$, for every $a\in A^\Delta$, $\nu\in\Omega_{\rm wk}^\Delta$, $s\in\Delta$; and
\item $\alpha(A^\Delta)$ is a rigid C$^*$-subalgebra of $B(\Omega_{\rm wk}^\Delta)$.
\end{enumerate}
That is, $( B(\Omega_{\rm wk}^\Delta),\alpha)$ is an injective envelope of $A^\Delta$.
\end{theorem}

We now arrive at the first main result of the present paper. Recall, from
Theorem \ref{a-komega-mloc inclusion}, that there is a monomorphism
$\beta:A\rightarrow K(\Omega^\Delta)=A^\Delta$.

\begin{theorem}\label{inj env} $\left(B(\Omega_{\rm wk}^\Delta), \alpha\circ\beta \right)$ is an injective envelope for $A$.
\end{theorem}

\begin{proof}
Theorem \ref{a-komega-mloc inclusion} asserts that
\[
\begin{CD}
A @>\beta >> K(\Omega^\Delta)
@>\gamma>> \mloc(A)
\end{CD}
\]
is a canonical embedding of $A$ into its local multiplier algebra.
Let $\iota_{\rm mloc}:\mloc(A)\rightarrow I\left(\mloc(A)\right)$
denote the canonical embedding of $\mloc(A)$ into its injective envelope.
By \cite[Theorem 5]{frank2002}, $\left( I(\mloc(A)), \iota_{\rm mloc}\circ\gamma\circ\beta\right)$ is an
injective envelope of $A$. Hence, by writing $I(A)=I\left(\mloc(A)\right)$, there exist embeddings
\begin{equation}\label{incl-1}
A\,\subset\,K(\Omega^\Delta)\, \subset \, \mloc (A)\, \subset \,I(A)\,,
\end{equation}
where the inclusions of $A$ into $\mloc(A)$ and $I(A)$ are the canonical inclusions.
Moreover, the inclusion of $K(\Omega^\Delta)$ into $I(A)$ is rigid because $K(\Omega^\Delta)$
contains $A$. Hence,
$\left(I(A), \kappa\right)$ is an injective envelope of $K(\Omega^\Delta)$, where $\kappa=\iota_{\rm mloc}\circ\gamma$.

If, for a given C$^*$-algebra $B$, $(C,\kappa)$ and $(\tilde C,\tilde\kappa)$ are two injective envelopes of $B$,
then there is an isomorphism $\varphi:C\rightarrow\tilde C$ such that $\varphi\circ\kappa=\tilde\kappa$ \cite[Theorem 4.1]{hamana1979a}.
Theorem \ref{ie stonean} asserts that $\left(B(\Omega_{\rm wk}^\Delta),\alpha\right)$ is an injective envelope of $K(\Omega^\Delta)$.
Hence,
\[
\begin{CD}
A @>\beta >> K(\Omega^\Delta) @>\iota_{\rm mloc}\circ\gamma >> I(A) \\
@.  @| @VV{\varphi}V \\
@. K(\Omega^\Delta)@>>{\alpha}> B(\Omega_{\rm wk}^\Delta)
\end{CD}
\]
for some isomorphism $\varphi$, which proves that
$\left(B(\Omega_{\rm wk}^\Delta), \alpha\circ\beta \right)$ is an injective envelope for $A$.
\end{proof}

%%%%%%%%%%%%%%%
\subsection{The second order local multiplier algebra}

\begin{theorem}\label{mloc-2} $\mloctwo(A)=\mloctwoplusk(A)=I(A)$ for all $k\in\mathbb N$.
\end{theorem}

\begin{proof} The injective algebra $I(A)=B(\Omega_{\rm wk}^\Delta)$ is a type I AW$^*$-algebra
and the ideal generated by the abelian projections of $I(A)$ is $K(\Omega_{\rm wk}^\Delta)$ \cite[Proposition 3.8]{argerami--farenick--massey2010}.
We will prove below that $e\in\mloc(A)$, for every abelian projection $e\in I(A)$. Assuming this statement holds, we therefore conclude that
$K(\Omega_{\rm wk}^\Delta)\subset\mloc(A)$. But $K(\Omega_{\rm wk}^\Delta)$ is an essential ideal of $I(A)$, and hence it  is also an essential ideal
of $\mloc (A)$. Therefore, $K(\Omega_{\rm wk}^\Delta)$ and $\mloc(A)$ have the same local multiplier algebras, which yields
$\mloctwo(A) =I(A)$ because of
\[
B(\Omega_{\rm wk}^\Delta)\;\supseteq\;\mloctwo(A)\;=\;\mloc(K(\Omega_{\rm wk}^\Delta))\;\supseteq\; M(K(\Omega_{\rm wk}^\Delta))
\;=\;B(\Omega_{\rm wk}^\Delta)\,.
\]
Hence, $\mloctwo(A)=\mloctwoplusk(A)=I(A)$, for all $k\in\mathbb N$.

Therefore, to complete the proof assume that $e\in I(A)$ and $\varepsilon>0$. Recall
that $e=\Theta_{\nu,\nu}$ for some $\nu\in \Omega_{\rm wk}^\Delta$ for which $\langle\nu,\nu\rangle$ is a projection
in $C(\Delta)$ \cite[Lemma 13]{kaplansky1953}. Because $\nu\in \Omega_{\rm wk}^\Delta$, there are a family $\{c_i\}_i$
of pairwise orthogonal projections in $C(\Delta)$ with supremum $1\in C(\Delta)$ and a bounded family $\{\omega_i\}_i\subset \Omega^\Delta$
such that $\|\nu-\xi\|<\varepsilon$ \cite[Proposition 4.4]{argerami--farenick--massey2010}, where
$\xi=\sum_ic_i\cdot\omega_i$ is in the sense of \eqref{aw sum} and $\|\xi\|<1+\varepsilon$.

By \eqref{incl-1}, $K(\Omega^\Delta)\subset\mloc(A)\subset B(\Omega_{\rm wk}^\Delta)$. Therefore, the centre of
$B(\Omega_{\rm wk}^\Delta)$, namely $\{f\cdot 1\,:\,f\in C(\Delta)\}$,  is contained in the centre of $\mloc(A)$ and
$\Theta_{\omega_i,\omega_i}\in\mloc(A)$ for all $i$. Thus, by
\cite[Lemma 3.3.6]{Ara--Mathieu-book} (see also \cite[Lemma 2.3 ]{somerset1996}),
\[
\mloc(A)\;=\;\displaystyle\prod_{i} c_i\mloc(A)\,,
\]
and under this isomorphism, $\left(c_i\cdot \Theta_{\omega_i,\omega_i}\right)_i$ determines a hermitian element $x\in\mloc(A)$.
Hence,
\[
\mloc(A) \;=\;\displaystyle\prod_{i} c_i\mloc(A)\;\subset\;\displaystyle\prod_{i} c_i I(A)\;=\; I(A)\,,
\]
where the last equality is a fact about AW$^*$-algebras \cite[Lemma 2.7]{kaplansky1951}. As $e\in I(A)$ is identified with $(c_i\cdot e)_i$
under this isomorphism, we obtain
\[
\|e-x\|\;=\;\sup_i\|c_i\left(\Theta_{\nu,\nu}-\Theta_{\omega_i,\omega_i}\right)\|
\;\leq\;\|\nu-\xi\|\left(\|\nu\|+\|\xi\|\right)\;<\;\varepsilon(1+\varepsilon)\,.
\]
Since $\varepsilon>0$ is arbitrary, $e$ is a limit of elements $x\in\mloc(A)$.
 \end{proof}

 Theorem \ref{mloc-2} demonstrates that the injectivity of $\mloctwo(A)$, which was proved to hold for separable type I C$^*$-algebras
 \cite[Theorem 2.7]{somerset2000} (see \cite[Theorem 3.2]{ara-mathieu2010} also), can hold as well for certain nonseparable type I C$^*$-algebras.
 In particular, the following special case of Theorem \ref{mloc-2} is new at this level of generality.

 \begin{corollary} \label{C_0(T)} For every  locally compact Hausdorff space $T$,  $\mloctwo(C_0(T)\otimes\mathbb K)$ is injective and therefore
 $\mloctwo(C_0(T)\otimes\mathbb K)=\mloctwoplusk(C_0(T)\otimes\mathbb K)$, for all $k\in\mathbb N$.
 \end{corollary}

%%%%%%%%%%%%%%%%%%%%%%%%%%%%%%%%%%%%%%%%%%%%
\subsection{A refinement of the chain of inclusions}
 
\begin{theorem}\label{mloc to ie1}
There exist monomorphisms through which the following inclusions as are as C$^*$-subalgebras:
\[
A\subset K(\Omega^\Delta)\subset \mloc(A)\subset \mloc(K(\Omega^\Delta))\subset \mloctwo(A)=\mloctwo(K(\Omega^\Delta))\,.
\]
\end{theorem}
\begin{lemma}\label{*-str cont} Let $\nu\in \mathcal E^I$ be such that $\check{\nu}\in C_0(X^I)$. If
$x\in M(I)$ is considered as strictly continuous bounded operator field on $X^I$, then there exists
$\omega\in\Omega^I$ such that $x(t)\nu(t)=\omega(t)$ for every $t\in X^I$.
\end{lemma}

\begin{proof} Let $\omega:\beta X^I\rightarrow\sqcup_{t\in\beta X^I}H_t^I$ be defined by $w(t)=0$ for $t\in\beta X^I\setminus X^I$
and $\omega(t)=x(t)\nu(t)$ for $t\in X^I$. We will show that for every $t_0\in\beta X^I$ and $\varepsilon>0$ there are an open
set $U\subset\beta X^I$ containing $t_0$ and a $\mu\in\Omega^I$ such that $\|\omega(t)-\mu(t)\|<\varepsilon$ for all $t\in U$.
Because $\Omega^I$ is closed under local uniform approximation, this will imply that $\omega\in\Omega^I$, thereby completing the
proof.

Assume $t_0\in \beta X^I$ and let $\varepsilon>0$. Notice that $\|x\|=\sup_{t\in X^I}\|x(t)\|<\infty$ and  $\|\omega(t)\|\leq \|x\| \,\check{\nu}(t)$ for $t\in X^I$. Thus, if $t_0\in \beta X^I\setminus X^I$ there exists an open set $t_0\in U\subset \beta X^I$ such that $\|\omega(t)\|\leq \epsilon$ for $t\in U$, since $\check{\nu}\in C_0(X^I)$.

Assume now that $t_0\in X^I$.  Choose a bounded vector field $\eta\in\Omega^I$ such that there exists an open set
$t_0\in W\subset X^I$ such that $\|\eta(t)\|=1$ for all $t\in W$.

Let $a=\Theta_{\nu,\eta}|_{X^I}$, which is
an element of $\mathcal F((\Omega^I_{|X^I})_0)$. Because $x$ is strictly continuous with respect to $\mathcal F((\Omega^I_{|X^I})_0)$,
there are an open set $U\subset W$ containing $t_0$ and $b\in\mathcal F((\Omega^I_{|X^I})_0)$ such that
$\|\left(x(t)-b(t)\right)a(t)\|<\varepsilon$ for all $t\in U$. Note that we have
\[
a(t)\eta(t)\,=\,\langle \eta(t),\eta(t)\rangle\,\nu(t)\,=\,\nu(t)\,,\;\forall\,t\in U\,.
\]
Let $\mu=b\nu\in\Omega^I$.
Hence, for any $t\in U$,
\[
\|\omega(t)-\mu(t)\|\,=\,
\|x(t)\nu(t)-b(t)\nu(t)\| \,=\,
\|\left(x(t)-b(t)\right)a(t)\eta(t)\|\;<\;\varepsilon\,.
\]
\end{proof}
\begin{proof}[Proof of Theorem \ref{mloc to ie1}] By Theorems \ref{a-komega-mloc inclusion} and \ref{mloc-2} we are left to show that there is a
monomorphism $\rho:\mloc(A)\rightarrow \mloc(A^\Delta)$, since $A^\Delta=K(\Omega^\Delta)$ by Proposition \ref{they are equal}.

To that end, let $I\in \mathcal I_{\rm ess}(A)$ and consider
the set $Y^I=\Phi_I^{-1}(X^I)\subset\Delta$ which is open and dense \cite[Lemma 1.1]{argerami--farenick--massey2009}.
Because $\Phi_I:\Delta\rightarrow \beta X^I$ is a (continuous) surjection $\Phi_I$ must map $Y^I$ onto $X^I$.
The  open dense set $Y^I$ determines an essential ideal of $A^\Delta$ that we denote by $\mathfrak h(I)$.
Thus, $\mathfrak h:  \mathcal I_{\rm ess}(A)\rightarrow  \mathcal I_{\rm ess}(A^\Delta)$ is a well defined function.
Note that if $K\in \mathcal I_{\rm ess}(A)$ is such that $K \preccurlyeq I$, then
\eqref{used a lot} states that
$X^I\subset\Phi_{KI}^{-1}(X^K)$
(because $\Phi_{KI}$ maps $\beta X^I\setminus X^I$ into $\beta X^K\setminus X^K$  \cite[Theorem 6.12]{Gillman-Jerison-book}). Thus,
\begin{equation}\label{incl inv-sets}
Y^K\;=\;\Phi_K^{-1}(X^K)\;=\;\Phi_I^{-1}\left( \Phi_{KI}^{-1}(X^K) \right) \;\supseteq\;Y^I\,,
\end{equation}
and so $\mathfrak h$ preserves order; i.e.,
$K \preccurlyeq I \;\Rightarrow\; \mathfrak h(K) \preccurlyeq \mathfrak h(I)\,.$

 Fix $I\in \mathcal I_{\rm ess}(A)$ and let $x\in M(I)$. Thus, by Proposition \ref{mult alg},
$x$ is a bounded cross
section of $(X^I,\{B(H^I_t)\}_{t\in X^I})$ which is strictly continuous with respect to $\cF((\Omega^I{}_{\vert X^I})_0)$.
Consider the bounded section $\tilde x=x\circ\Phi_I{}_{\vert Y^I}$ of the fibred space
$(Y^I,\{B(H^\Delta_s)\}_{s\in Y^I})$. We aim to show that $\tilde x$ is strictly continuous with respect to
$\cF((\Omega^\Delta{}_{\vert Y^I})_{0})$, as this is sufficient (and necessary) for $\tilde x\in M(\mathfrak h(I))$ by Proposition \ref{mult alg}. To this end, let $s_0\in Y^I$, $\varepsilon>0$, and $a\in \cF((\Omega^\Delta{}_{\vert Y^I})_{0})$. Recall that
$\Delta=\displaystyle\lim_{\leftarrow}\beta X^K$ and, by Theorem \ref{omega limit},  $\Omega^\Delta\;=\;\displaystyle\lim_{\rightarrow}\,\Omega^K$. Thus, without loss of generality we can assume that
 there are an essential ideal $K\subset A$ with $K\subset I$, an open set $U\subset \beta X^K$ with $s_0\in\Phi_K^{-1}(U)\subset Y^I$ and $\omega_j,\eta_j\in  \mathcal E^K$ such that $$a(s)=\displaystyle\sum_{j=1}^n\Theta_{\omega_j,\eta_j}\circ\Phi_K(s) \ \text{ for }\ s\in \Phi_K^{-1}(U)\,,$$ since
 $\Omega^K$ consists of all vector fields $\nu:\beta X^K \rightarrow \bigsqcup_{t\in \beta X^K}\,H_t^K$ that are local uniform limits of $\mathcal E^K$. Again, since the strict continuity is a local property, we can further assume that $\check{\omega_j},\,\check{\eta_j}\in C_0(U)$.

Within the open subset $U\subset\beta X^K$, apply the Gram--Schmidt orthogonalisation procedure \cite[Lemma 4.2]{fell1961}
to the vector fields $\omega_j,\eta_j\in \mathcal E^K$ to obtain vector fields $\nu_1,\dots,\nu_N\in\mathcal E^K$
that are pairwise orthogonal in an open set $t_0=\Phi_K(s_0)\in U_0\subset U$ and are such that each $\omega_j(t)$ and $\eta_j(t)$
are in the linear span of $\nu_1(t),\dots,\nu_N(t)$ for every $t\in U_0$, $1\leq j\leq n$. (Notice that $\mathcal E^K$ is a $C_b(T)$-module via the natural monomorphism from $C_b(T)$ into $C(\beta X^K)$; this is all that is needed for the Gram--Schmidt process.)
Relabel so that $U$ now has the property of $U_0$.

Because $I \preccurlyeq K$, the proof of
Proposition \ref{systems1} demonstrates that the map $\omega\mapsto\omega\circ\Phi_{IK}$
is a linear isomorphism  $\mathcal E^K\rightarrow\mathcal E^I$, allowing one to go back and forth between $\mathcal E^K$ and $\mathcal E^I$. Hence, we may further assume that the vector fields
$\omega_j,\eta_j, \nu_\ell\in \mathcal E^K$ are contained in $\mathcal E^I$ and defined on $\beta X^I$ and are such that $\check{\omega_j},\,\check{\eta_j}\in C_0(U)\subset C_0(X^I)$ (since $U\subset X^K\subset X^I$). Now let $p=\sum_{i=1}^N\Theta_{\nu_i,\nu_i}\in A^I$. By Lemma \ref{*-str cont}, each of
\[
px=\sum_{i=1}^N \Theta_{\nu_i,\,x^*\nu_i}\ , \; xp=\sum_{i=1}^N \Theta_{x\nu_i,\, \nu_i}\ , \; \mbox{and } pxp=\sum_{i=1}^N \Theta_{x\nu_i,\,x^*\nu_i}
\]
can naturally be regarded as an element of $A^I$.
Notice that $p(t)$ is the orthogonal projection onto the span of $\{\nu_1(t),\ldots,\nu_N(t)\}$ for every $t\in U$. Let $c=(px+xp-pxp)\circ \Phi_{IK}\in A^K$ and let
$d=\sum_{i=1}^n \Theta_{\omega_i,\eta_i}$ whereby $a=d\circ \Phi_K$.
Hence $d(x-c )(t)=dx(t)-dx(t)=0$ for $t\in U$, since
\[
dc(t)=dpx(t)+dxp(t)-dpxp(t)=dx(t)+dpxp(t)-dpxp(t) \,,
\]
because $d(t)=dp(t)=pd(t)$ for $t\in U$. Similarly $(x-c)d(t)=xd(t)-xd(t)=0$ for $t\in U$.

If we now let $b=c\circ \Phi_K$ then $b\in \cF((\Omega ^\Delta|_{Y^I})_0)$ - since $\check c\in C_0(U)$ - is such that
\[
\|\left( \tilde x(s)-b(s)\right)a(s)\|\,+\,\|a(s)\left( \tilde x(s)-b(s)\right)\|\;<\;\varepsilon\,,\;\mbox{ for all }\;s\in \Phi_K^{-1}(U)\subset Y^I\,.
\]
This proves that $\tilde x\in M(\mathfrak I(I))$.
The map $\zeta_I:M(I)\rightarrow M(\mathfrak h(I))$ given by $\zeta_I(x)=\tilde x$ is evidently a homomorphism.
If $\zeta_I(x)=0$, then $x(t)=0$ for all $t\in \Phi_I(Y^I)=X^I$, and so $x=0$. Therefore, $\zeta_I$ is a monomorphism.
Let $\alpha_{\mathfrak h(I)}:M(\mathfrak h(I))\rightarrow \mloc(A^\Delta)$ be the unique monomorphism that embeds $M(\mathfrak h(I))$
into the local multiplier algebra of $A^\Delta$ and, for $J \preccurlyeq I$, let $\alpha_{\mathfrak h(J)\mathfrak h(I)}:M(\mathfrak h(J))\rightarrow M(\mathfrak h(I))$
be the connecting monomorphisms induced by $\mathfrak h(J) \preccurlyeq \mathfrak h(I)$.
For each  $I\in \mathcal I_{\rm ess}(A)$, let
 $\rho_I:M(I)\rightarrow \mloc(A^\Delta)$ be the monomorphism $\rho_I=\alpha_{\mathfrak h(I)}\circ\zeta_I$.
Because $\tilde\pi_{JI}=\alpha_{\mathfrak h(J)\mathfrak h(I)}\circ\zeta_J$ (where $\tilde\pi_{JI}$ is as in \eqref{compat1}) we conclude that the following diagram
 \[
\begin{CD}
M(J) @>{\tilde\pi_{JI}}>> M(I) \\
@V{\zeta_J}VV           @VV{\zeta_I}V  \\
M(\mathfrak h(J))  @>>{\alpha_{\mathfrak h(J)\mathfrak h(I)}}>  M(\mathfrak h(I)) \\
\end{CD}
\]
is commutative.
Therefore, there exists a monomorphism $\rho:\mloc(A)\rightarrow \mloc(A^\Delta)$ by the universal property of $\mloc(A)$. 
\end{proof}
%%%%%%%%%%%%%%%%%%%%%%%%%%%%%%%%%%
\section*{Acknowledgement}
This work was undertaken with the support of the
NSERC Discovery Grant Program (Canada) and PIP--CONICET (Argentina). Part of the work of the second
author was undertaken at the Institut de Math\'ematiques de Jussieu, Paris, and he wishes to thank IMJ for its
hospitality and support during his stay.
%%%%%%%%%%%%%%%%%%%%%%%%%%% bibliography %%%%%%%%%%%%%%%%%%%%%%%%%
\bibliographystyle{abbrv}
\bibliography{wchb}
\end{document}